\documentclass{amsart}
\usepackage{tikz}
\usepackage{hyperref}   
\usepackage{amssymb,xcolor}
\usepackage{stmaryrd} 
\usepackage{enumerate}

\usepackage{mathtools}  
\mathtoolsset{showonlyrefs}

\theoremstyle{plain}
\newtheorem{theorem}{Theorem}[section]
\theoremstyle{plain}

\theoremstyle{plain}
\newtheorem{definition}[theorem]{Definition}
\theoremstyle{plain}
\newtheorem{lemma}[theorem]{Lemma}
\theoremstyle{remark}
\newtheorem{remark}[theorem]{Remark}
\theoremstyle{plain}
\newtheorem{proposition}[theorem]{Proposition}
\numberwithin{equation}{section}
\theoremstyle{plain}

\begin{document}

\title{Initial data set rigidity results for polyhedra}

\author{Xiaoxiang Chai}
\address{Department of Mathematics, POSTECH, 77 Cheongam-Ro, Nam-Gu, Pohang,
Gyeongbuk, Korea 37673}
\email{xxchai@kias.re.kr, xxchai@postech.ac.kr}

\author{Xueyuan Wan}
\address{Mathematical Science Research Center, Chongqing University of
Technology, Chongqing 400054, China}
\email{xwan@cqut.edu.cn}


\begin{abstract}
  Using spinors, we show a dihedral type rigidity for polyhedral initial data
  sets. This rigidity connects spacetime positive mass theorem, dihedral
  rigidity and capillary marginally outer trapped surfaces. Our method is to
  extend the rigidity analysis of spacetime positive mass theorem due to
  Beig-Chrusciel to the settings of a twisted spinor bundle.
\end{abstract}

{\maketitle}

\section{Introduction}

Let $(M, g)$ be an $n$-dimensional Riemannian manifold with non-empty boundary
$\partial M$ and $q$ be a symmetric 2-tensor. The triple $(M, g, q)$ is called
an \text{{\itshape{initial data set}}}. If the manifold $(M, g)$ is spin, then
we call $(M, g, q)$ a spin initial data set. The \text{{\itshape{energy
density}}} $\mu$ and \text{{\itshape{current density}}} $J$ are defined by
\begin{align}
2 \mu & = R_g + (\ensuremath{\operatorname{tr}}_g q)^2 - |q|_g^2, \label{mu}
\\
J & =\ensuremath{\operatorname{div}}_g q - \mathrm{d}
(\ensuremath{\operatorname{tr}}_g q), \label{J}
\end{align}
where $R_g$ is the scalar curvature of $(M, g)$. The initial data set $(M, g,
q)$ is said to satisfy the \text{{\itshape{dominant energy condition}}} if
\begin{equation}
  \mu \geqslant |J| . \label{eq:dec}
\end{equation}
This condition is a generalization of the non-negative scalar curvature. A
celebrated result in scalar curvature geometry is Schoen-Yau's solution
{\cite{schoen-existence-1979}} to the Geroch conjecture which states that the
three dimensional torus does not admit a metric of non-negative scalar
curvature except the flat metric (see also {\cite{gromov-positive-1983}}).

A generalization of the Geroch conjecture was given by
Eichmair-Galloway-Mendes {\cite{eichmair-initial-2021}} regarding to the
rigidity of over-torical initial data sets.

Let $\Sigma$ be a two-sided hypersurface in a Riemannian manifold $(M, g)$ and
$\nu$ be a unit normal to $\Sigma$ and $h = \nabla^g \nu$ be the second
fundamental form of $\Sigma$, $H$ be the mean curvature which is the trace of
$h$. The 2-form $Q^{\pm} = h \pm q|_{\Sigma}$ on $\Sigma$ is called the
\text{{\itshape{null second fundamental form}}} with respect to $\nu$. The
trace $\ensuremath{\operatorname{tr}}_{\Sigma} Q^{\pm} = H \pm
\ensuremath{\operatorname{tr}}_{\Sigma} q$ of $Q^{\pm}$ over the tangent space
of $\Sigma$ is called the \text{{\itshape{null expansion}}}. A marginally
outer (inner) trapped hypersurface $\Sigma$ (see
{\cite{andersson-local-2005}}) in $(M, g, q)$ is a hypersurface which
satisfies $H \pm \ensuremath{\operatorname{tr}}_{\Sigma} q = 0$. For
convenience, we call $\Sigma$ a MOTS (MITS) in short. Using MOTS and its
stability, Eichmair-Galloway-Mendes proved the following rigidity result.
(Note that we have put $M =\mathbb{T}^{n - 1} \times [- 1, 1]$ for
simplicity.)

\begin{theorem}[{\cite[Theorem 1.2]{eichmair-initial-2021}}]
  \label{thm:egm}Let $(M =\mathbb{T}^{n - 1} \times [- 1, 1], g, q)$, $3
  \leqslant n \leqslant 7$ be an over-torical initial data set and
  $\partial_{\pm} M =\mathbb{T}^{n - 1} \times \{\pm 1\}$. If
  $\ensuremath{\operatorname{tr}}_{\partial_+ M} Q^+ \geqslant 0$ on
  $\partial_+ M$ with respect to the outward unit normal and
  $\ensuremath{\operatorname{tr}}_{\partial_{- M}}  Q^+ \leqslant 0$ on
  $\partial_- M$ with respect to the inward unit normal of $\partial_- M$,
  then:
  \begin{enumerate}[a{\textup{)}}]
    \item The manifold $M$ is diffeomorphic to $\partial_- M \times [0, l]$
    for some $l > 0$. Let $\Sigma_t = \partial_- M \times \{t\}$, $t \in [0,
    l]$ be the level set, every level set is a flat torus.
    
    \item Let $\nu_t$ be the unit normal of $\Sigma_t$ pointing to the $t$
    direction, then the null second fundamental form $Q^+$ of $\Sigma_t$ with
    respect to $\nu_t$ in $(M, g, q)$ vanishes.
    
    \item The identity $\mu + J (\nu_t) = 0$ along every $\Sigma_t$.
  \end{enumerate}
\end{theorem}

An observation made by Lohkamp
{\cite{lohkamp-scalar-1999,lohkamp-higher-2016}} is that Theorem \ref{thm:egm}
implies the spacetime positive mass theorem {\cite{schoen-proof-1981}},
{\cite{witten-new-1981}}, {\cite{eichmair-jang-2013}},
{\cite{eichmair-spacetime-2016}}. \ Under the assumption of negative mass, his
idea is to reduce the asymptotically flat initial data set to an initial data
set which is exactly flat and with vanishing $q$ outside a large cube, but
with strict positive $\mu - |J|$ in a smaller subset of the cube, and then to
identify $(n - 1)$ pairs of edges to obtain an over-torical initial data set.
The reduction to such an over-torical initial data set contradicts Theorem
\ref{thm:egm}. Hence, it leads to an alternative proof of the spacetime
positive mass theorem.

Gromov {\cite{gromov-dirac-2014}} proposed the dihedral rigidity conjecture to
characterize the scalar curvature lower bound in the weak sense. The dihedral
rigidity turns out also closely related to the positive mass theorem. Indeed,
it states that the only metric on a convex Euclidean polyhedron of
non-negative scalar curvature, with mean convex faces and less dihedral angles
than that of the Euclidean metric. Assuming negative mass and performing
Lohkamp's reduction without the identification of the edges, we obtain a cube
contradicting the dihedral rigidity. See also {\cite{miao-measuring-2020}},
{\cite{miao-mass-2022}} and {\cite{jang-hyperbolic-2021}}.

The dihedral rigidity conjecture was more difficult than the Geroch conjecture
due to singularities, but we still have seen much progress, for example,
{\cite{gromov-dirac-2014}}, {\cite{li-polyhedron-2020}},
{\cite{li-dihedral-2020-1}}, {\cite{brendle-scalar-2024}},
{\cite{wang-gromovs-2022-arxiv}}, {\cite{chai-dihedral-2023}},
{\cite{chai-scalar-2024}}. The list given here is not exhaustive.

Motivated by Gromov dihedral rigidity conjecture and Theorem \ref{thm:egm}, we
investigate the rigidity for polyhedral initial data sets. Let $\Omega$ be a
compact, convex polyhedron in $\mathbb{R}^n$ with non-empty interior. We may
write $\Omega = \cap_{\ell \in \Lambda} \{u_{\ell} \leqslant 0\}$, where
$u_{\ell}$, $\ell \in \Lambda$ is a finite collection of non-constant linear
functions. Let $g$ be another Riemannian metric which is defined on an open
set containing $\Omega$. For each $\ell \in \Lambda$, we denote $N_{\ell} \in
\mathbb{S}^{n - 1}$ the outward unit normal to the half-space $\{u_{\ell}
\leqslant 0\}$ and $F_{\ell} = \partial \Omega \cap \{u_{\ell} \leqslant 0\}$
be a face of the polyhedron $F_{\ell}$, $\nu_{\ell}$ the unit outward unit
normal to $F_{\ell}$ with respect to the metric $g$. We introduce
\text{{\itshape{tilted dominant energy condition}}} on the face $F_{\ell}$,
\begin{equation}
  H_{\ell} + \cos \theta_{\ell} \ensuremath{\operatorname{tr}}_{F_{\ell}} q
  \geqslant \sin \theta_{\ell} |q (\nu_{\ell}, \cdot)^{\top} | \label{tilt
  dec}
\end{equation}
where $\theta_{\ell} \in [0, \pi]$ is given by $\cos \theta_{\ell} = \langle
\tfrac{\partial}{\partial x^1}, N_{\ell} \rangle$ and $q (\nu_{\ell},
\cdot)^{\top}$ is the component tangential to $F_{\ell}$ of the 1-form $q
(\nu_{\ell}, \cdot)$. The cases $\theta_{\ell} = 0, \pi / 2, \pi$ were
introduced by {\cite{almaraz-spacetime-2021}} and further generalized to any
$\theta_{\ell} \in [0, \pi]$ by the first author {\cite{chai-tilted-2023}}.

Our main result is the following initial data set rigidity for polyhedra.

\begin{theorem}
  \label{dihedral ids}Assume $\Omega$ is a convex polyhedron in
  $\mathbb{R}^n$, $g$ be a Riemannian metric on $\Omega$ and $q$ a symmetric
  2-tensor ($g$ and $q$ are defined in an open set containing $\Omega$).
  Assuming $(\Omega, g, q)$ satisfies dominant energy condition, tilted
  dominant energy condition; moreover, if $p \in \partial \Omega$ such that
  there are two indices $\ell_1$, $\ell_2 \in \Lambda$ satisfying $u_{\ell_1}
  (p) = u_{\ell_2} (p) = 0$, then $g (\nu_{\ell_1}, \nu_{\ell_2}) = \langle
  N_{\ell_1}, N_{\ell_2} \rangle_{\delta}$ at the point $p$ (matching angle).
  Then the following hold for $(\Omega, g, q)$:
  \begin{enumerate}[a{\textup{)}}]
    \item The polyhedron $(\Omega, g)$ is foliated by level sets $\Sigma_t$ of
    a function $t$. Each level set $\Sigma_t$ is a flat polyhedron.
    
    \item Every leaf $\Sigma_t$ intersects each face $F_{\ell}$ at constant
    angle $\theta_{\ell}$.
    
    \item Let $\nu_t$ be the unit normal pointing to the $t$ direction, the
    null second fundamental form $Q^+$ with respect to $\nu_t$ of $\Sigma_t$
    in $(\Omega, g, q)$ vanishes.
    
    \item Let $\{e_i \}$ be an orthonormal frame such that $e_n = \nu_t$ and
    \[ \hat{R}_{i j k l} = R_{i j k l} + q_{j k} q_{i l} - q_{i k} q_{j l} .
    \]
    Then
\begin{align}
\nabla_i q_{j n} - \nabla_j q_{i n} & = 0, \\
\hat{R}_{i j k n} & = \nabla_i q_{j k} - \nabla_j q_{i k} \text{ for all
} i, j, k, \\
\hat{R}_{i j k l} & = 0 \text{ for } k < n, l < n.
\end{align}
    In particular, $\mu + J (e_n) = 0$ along $\Sigma_t$.
  \end{enumerate}
  The convention for $R_{i j k l}$ that we use is
  \[ R_{i j k l} = \langle \nabla_i \nabla_j e_k, e_l \rangle - \langle
     \nabla_j \nabla_i e_k, e_l \rangle - \langle \nabla_{[e_i, e_j]} e_k, e_l
     \rangle . \]
\end{theorem}

As mentioned before, Theorem \ref{dihedral ids} implies the spacetime positive
mass theorem by Lohkamp's reduction
{\cite{lohkamp-scalar-1999,lohkamp-higher-2016}}. Alternatively, assuming
negative mass and dominant energy condition \eqref{eq:dec}, the asymptotics of
the asymptotically flat initial data can be reduced to harmonic asymptotics
with strict $\mu > |J|$. The the tilted dominant energy condition \eqref{tilt
dec} can be verified on sufficiently large polyhedra with $N_0$ being the
direction of the total linear momentum (see {\cite[Lemma
5]{eichmair-spacetime-2016}}). The incompatibility with Theorem \ref{dihedral
ids} proves the spacetime positive mass theorem.

Theorem \ref{dihedral ids} is a generalization of several previous results,
for example {\cite{li-polyhedron-2020}}, {\cite{li-dihedral-2020-1}},
{\cite{brendle-scalar-2024}}, {\cite{wang-gromovs-2022-arxiv}},
{\cite{chai-dihedral-2023}}, {\cite{chai-scalar-2024}}, of dihedral rigidity
to initial data sets. Our method is spinorial and in the same vein as
Beig-Chrusciel {\cite{beig-killing-1996}} in the rigidity of the spacetime
positive mass theorem. The key difference is that we work on a twisted spinor
bundle. Although it takes extra work, working on a twisted spinor bundle
allows us to obtain parallel spinors of full dimensions on every leaf.

As for the problems of boundary singularities, we use Brendle's smoothing
procedure {\cite{brendle-scalar-2024}} which requires the matching angles. We
use the boundary condition \eqref{tilt dec} instead of $H_{F_{\ell}} \geqslant
|\ensuremath{\operatorname{tr}}_{F_{\ell}} q|$ used in
{\cite{brendle-rigidity-2023}}, since the condition \eqref{tilt dec} also
covers the hyperbolic case {\cite{li-dihedral-2020-1}},
{\cite{chai-dihedral-2023}}, {\cite{chai-scalar-2024}} where $q = \pm g$.

For cubical initial data sets, another approach via harmonic functions in
dimension 3 was studied by {\cite{tsang-dihedral-2021-arxiv}} (see also
{\cite{chai-scalar-2022}}). Several previous works {\cite{gromov-dirac-2014}},
{\cite{li-polyhedron-2020}}, {\cite{li-dihedral-2020-1}},
{\cite{chai-dihedral-2023}} and that the polyhedral initial data set $(\Omega,
g, q)$ in Theorem \ref{dihedral ids} is foliated by capillary MOTS suggest
that \text{{\itshape{capillary MOTS}}} (a MOTS which intersects faces
$F_{\ell}$ of $\Omega$ at a constant angle $\theta_{\ell}$, see
{\cite{alaee-stable-2020}}) might be a viable tool to study the rigidity of
polyhedral initial data sets. However, the existence and regularity theory as
{\cite{andersson-area-2009,eichmair-plateau-2009}} of capillary MOTS in a
manifold with polyhedral boundaries can be very challenging.

Motivated by Theorem \ref{thm:egm}, using non-spin methods, we prove some
relations of the curvature operators of $(M, g)$ and the 2-tensor $q$ in
similar ways as in Theorem \ref{dihedral ids}. However, we are not able to
prove all the curvature and $q$ relations appeared in Theorem \ref{thm:egm}
(see Remark \ref{not full relation}). We tend to believe that the relations in
Theorem \ref{dihedral ids} should be valid for over-torical initial data sets
as well.

We would like to remark that these relations were also found by Hirsch-Zhang
{\cite{hirsch-case-2022}}, {\cite{hirsch-initial-2024}} for asymptotically
flat initial data sets via a mixed usage of spacetime harmonic functions and
spinors. We do not use spacetime harmonic functions. For the analysis of
rigidity of spacetime positive mass theorem for asymptotically flat initial
data sets using non-spin methods, we refer to a series of papers by Huang-Lee
{\cite{huang-equality-2020}}, {\cite{huang-equality-2023-ii}},
{\cite{huang-bartnik-2024}}.

\

The article is organized as follows:

\

In Section \ref{sec:geometric construction}, we construct a twisted spinor
bundle, a Dirac operator and derive related Schrodinger-Lichnerowicz formula.
In Section \ref{sec:ids rigidity for polyhedra}, we solve a Dirac type
equation and perform the rigidity analysis. In Section \ref{sec:extra}, we
prove further results based on Theorem \ref{thm:egm} and comment how to handle
the odd dimensional case.

\

\text{{\bfseries{Acknowledgments.}}} X. Chai was supported by the National
Research Foundation of Korea (NRF) grant funded by the Korea government (MSIT)
(No. RS-2024-00337418) and an NRF grant No. 2022R1C1C1013511. X. Wan was
supported by the National Natural Science Foundation of China (Grant No.
12101093) and the Natural Science Foundation of Chongqing (Grant No.
CSTB2022NSCQ-JQX0008), the Scientific Research Foundation of the Chongqing
University of Technology.

\

\section{Geometric construction}\label{sec:geometric construction}

In this section, we construct a twisted spinor bundle on which we define a
modified connection and a Dirac operator. We prove an integrated form of a
Schrodinger-Lichnerowicz formula.

\subsection{Construction of a twisted spinor bundle}\label{sub:construct}

Let $M$ be a smooth domain in $\mathbb{R}^n$, $(M, g, q)$ be an initial data
set, $S_{M_g}$ ($S_{M_{\delta}}$) be the spinor bundle associated with the
metric $g$ ($\delta$). Let $\nabla^g$ $(\bar{\nabla})$ be the Levi-Civita
connection with respect to the metric $g$ ($\delta$). Then the connection
\[ \nabla = \nabla^g \otimes 1 + 1 \otimes \bar{\nabla} \]
is a natural connection on the twisted spinor bundle $S_{M_g} \otimes
S_{M_{\delta}}$. Let $\{E_i \}$ be an orthonormal frame with respect to the
Euclidean metric. We denote the Clifford multiplication of vector $v$
($\bar{v}$) with respect to the metric $g$ by $c (v)$ ($\bar{c} (\bar{v})$).
The Dirac operator $D$ associated with $\nabla$ is
\[ D = \sum_{i = 1}^n c (e_i) \nabla_{e_i} \]
where $\{e_i \}$ is a local orthonormal frame with respect to the metric $g$.
Let $\Sigma$ be a two-sided hypersurface in $M$, we define the hypersurface
(or boundary) Clifford multiplications by
\begin{equation}
  \bar{c}_{\partial} (e_i) = \bar{c} (e_i) \bar{c} (E_n), \text{ }
  c_{\partial} (e_i) = c (e_i) c (e_n) \label{boundary clifford}
\end{equation}
where $E_n$ is a unit normal to $\Sigma$ in $(M, \delta)$ and $e_n$ is a unit
normal to $\Sigma$ in $(M, g)$. The hypersurface (or boundary) spinor
connections are given by
\begin{align}
\nabla_{e_i}^{g, \partial} & = \nabla_{e_i}^g + \tfrac{1}{2} c
(\nabla_{e_i}^g e_n) c (e_n), \\
\bar{\nabla}_{e_i}^{\partial} & = \bar{\nabla}_{e_i} + \tfrac{1}{2} \bar{c}
(\bar{\nabla}_{e_i} E_n) \bar{c} (E_n), \label{Euclid boundary connection}
\end{align}
and
\begin{equation}
  \nabla^{\partial} = \nabla^{g, \partial} \otimes 1 + 1 \otimes
  \bar{\nabla}^{\partial} . \label{boundary connection}
\end{equation}
The boundary Dirac operator is given by
\begin{equation}
  D^{\partial} = \sum_{j = 1}^{n - 1} c_{\partial} (e_j)
  \nabla_{e_j}^{\partial} . \label{boundary Dirac}
\end{equation}
\begin{remark}
  \label{notation abuse}For our application, it is useful to consider a vector
  field $E_n$ which is neither tangent nor normal to $\Sigma$. For such $E_n$,
  we still use the notations introduced in \eqref{boundary clifford},
  \eqref{Euclid boundary connection}, \eqref{boundary connection} and
  \eqref{boundary Dirac}.
\end{remark}

Assume that the dimension of $M$ is even, then there exists a
$\mathbb{Z}_2$-grading $\varepsilon \otimes \bar{\varepsilon}$ of $S_{M_g}
\otimes S_{M_{\delta}}$ where $\varepsilon$ and $\bar{\varepsilon}$ are the
$\mathbb{Z}_2$-grading of $S_{M_g}$ and $S_{M_{\delta}}$ respectively. In
fact, $\varepsilon = (\sqrt{- 1})^{n / 2} c (e_1) c (e_2) \cdots c (e_n)$ and
$\bar{\varepsilon} = (\sqrt{- 1})^{n / 2} \bar{c} (E_1) \bar{c} (E_2) \cdots
\bar{c} (E_n)$.

We fix a unit Euclidean vector $N_0 = \tfrac{\partial}{\partial x^1}$ and we
define a modified connection $\hat{\nabla}$ on $S_{M_g} \otimes
S_{M_{\delta}}$ by
\begin{equation}
  \hat{\nabla}_{e_i} \sigma = \nabla_{e_i} \sigma + \tfrac{1}{2} (\varepsilon
  \otimes \bar{\varepsilon}) (c (q_{i j} e_j) \otimes \bar{c} (N_0)) \sigma
  \label{dw connection}
\end{equation}
where $\sigma$ is a section of $S_{M_g} \otimes S_{M_{\delta}}$ and we have
used the Einstein summation convention and the Dirac operator is
\begin{equation}
  \hat{D} \sigma = D \sigma - \tfrac{1}{2} \ensuremath{\operatorname{tr}}_g q
  (\varepsilon \otimes \bar{\varepsilon}) \bar{c} (N_0) \sigma . \label{dw
  operator}
\end{equation}
We often also denote $q (e_i) = q_{i j} e_j$.

\begin{definition}
  A section $\sigma$ of $S_{M_g} \otimes S_{M_{\delta}}$ is said to satisfy
  the local boundary condition $B$ if
  \begin{equation}
    \chi \sigma := (\varepsilon \otimes \bar{\varepsilon}) (c (e_n) \otimes
    \bar{c} (N)) \sigma = \sigma \label{chi}
  \end{equation}
  where $N : \partial M \to \mathbb{S}^{n - 1}$ be a smooth map homotopic to
  the Euclidean Gauss map of $\partial M$. Note that $N$ is not necessarily a
  unit normal of $\partial M$ in $(M, \delta)$.
\end{definition}

\subsection{Schrodinger-Lichnerowicz formula}

We now present the integrated form of Schrodinger-Lichnerowicz formula.

\begin{proposition}
  \label{sl no B}Let $\sigma$ be a section of the twisted spinor bundle
  $S_{M_g} \otimes S_{M_{\delta}}$, $N : \partial M \to \mathbb{S}^{n - 1}$ be
  a smooth map homotopic to the Euclidean Gauss map of $\partial M$ and set
  $E_n = N$ in \eqref{Euclid boundary connection}, \eqref{boundary connection}
  and \eqref{boundary Dirac} (see Remark \ref{notation abuse}), then
  \[ D^{\partial} \chi + \chi D^{\partial} = 0, \]
  (that is, $D^{\partial}$ and $\chi$ anti-commute) and
\begin{align}
& \int_M | \hat{D} \sigma |^2 \\
= & \int_M [| \hat{\nabla} \sigma |^2 + \langle \sigma, \mu \sigma +
(\varepsilon \otimes \bar{\varepsilon}) c (J^{\sharp}) \otimes \bar{c}
(N_0) \sigma \rangle] \\
& \quad + \int_{\partial M} [\tfrac{1}{4} \langle D^{\partial} (\sigma +
\chi \sigma), \sigma - \chi \sigma \rangle + \tfrac{1}{4} \langle
D^{\partial} (\sigma - \chi \sigma), \sigma + \chi \sigma \rangle]
\\
& \quad + \int_{\partial M} \langle \mathcal{A} \sigma, \sigma \rangle +
\tfrac{1}{2} \int_{\partial M} \langle (\varepsilon \otimes
\bar{\varepsilon}) c ((\ensuremath{\operatorname{tr}}_g q) e_n - q (e_n))
\bar{c} (N_0) \sigma, \sigma \rangle, \label{eq:Sl}
\end{align}
  where $J^{\sharp}$ is the dual vector of the 1-form $J$ and
  \begin{equation}
    \mathcal{A}= \tfrac{1}{2} H + \tfrac{1}{2} \sum_{i \neq n} c (e_n) c (e_i)
    \bar{c} (\bar{\nabla}_{e_i} N) \bar{c} (N) . \label{A}
  \end{equation}
\end{proposition}

\begin{proof}
  The proof of that $D^{\partial}$ and $\chi$ anti-commute is a tedious but
  direct calculation by using only (anti-)commutative properties of Clifford
  multiplication and the Definition of $D^{\partial}$ in \eqref{boundary
  Dirac}. We set
  \begin{equation}
    \Psi = \tfrac{1}{2} \ensuremath{\operatorname{tr}}_g q (\varepsilon
    \otimes \bar{\varepsilon}) \bar{c} (N_0), \label{Psi}
  \end{equation}
  so $\hat{D} = D + \Psi$. For simplicity, we assume that $\{e_i \}$ is a
  geodesic orthonormal frame to simplify computations. We have
\begin{align}
| \hat{D} \sigma |^2 & = |D \sigma |^2 + \langle D \sigma, \Psi \sigma
\rangle + \langle \Psi \sigma, D \sigma \rangle + \langle \Psi \sigma,
\Psi \sigma \rangle \\
& = |D \sigma |^2 + \langle D \sigma, \Psi \sigma \rangle + \langle \Psi
\sigma, D \sigma \rangle + \tfrac{1}{4} (\ensuremath{\operatorname{tr}}_g
q)^2 | \sigma |^2 .
\end{align}
  From integration by parts,
\begin{align}
\int_M |D \sigma |^2 & = \int_M (\nabla_i \langle D \sigma, c (e_i) \sigma
\rangle - \langle \nabla_i (D \sigma), c (e_i) \sigma \rangle) \\
& = \int_{\partial M} (\langle D \sigma, c (e_n) \sigma \rangle + \langle
D^2 \sigma, \sigma \rangle) \\
& = - \int_{\partial M} \langle c (e_n) D \sigma, \sigma \rangle - \int_M
\langle \nabla_i \nabla_i \sigma, \sigma \rangle + \int_M \tfrac{1}{4} R_g
| \sigma |^2 .
\end{align}
  In the last line we have used the Schrodinger-Lichnerowicz formula (see
  {\cite{gromov-positive-1983}})
  \[ D^2 = \nabla^{\ast} \nabla + \tfrac{1}{4} R_g = - \nabla_i \nabla_i
     \sigma + \tfrac{1}{4} R_g \]
  on a twisted spinor bundle and that $S_{M_{\delta}}$ is a flat spinor
  bundle. By integration by parts on the term containing $\nabla_i \nabla_i
  \sigma$,
  \[ \int_M \langle \nabla_i \nabla_i \sigma, \sigma \rangle = \int_M
     (\nabla_i \langle \nabla_i \sigma, \sigma \rangle - | \nabla \sigma |^2)
     = \int_{\partial M} \langle \nabla_{e_n} \sigma, \sigma \rangle - \int_M
     | \nabla \sigma |^2 . \]
  Hence
  \[ \int_M |D \sigma |^2 = \int_M (| \nabla \sigma |^2 + \tfrac{1}{4} R_g |
     \sigma |^2) - \int_{\partial M} (\langle c (e_n) D \sigma + \nabla_{e_n}
     \sigma, \sigma \rangle) . \]
  We replace now $| \nabla \sigma |^2$ by $| \hat{\nabla} \sigma |^2$ using
  the following
\begin{align}
& | \nabla \sigma |^2 \\
= & \left| \hat{\nabla}_i \sigma - \tfrac{1}{2} (\varepsilon \otimes
\bar{\varepsilon}) c (q (e_i)) \otimes \bar{c} (N_0) \sigma \right|^2
\\
= & | \hat{\nabla} \sigma |^2 + \tfrac{1}{4} |q|_g^2 | \sigma |^2 -
\tfrac{1}{2} \langle \hat{\nabla}_i \sigma, (\varepsilon \otimes
\bar{\varepsilon}) c (q (e_i)) \otimes \bar{c} (N_0) \sigma \rangle
\\
& \quad - \tfrac{1}{2} \langle (\varepsilon \otimes \bar{\varepsilon}) c
(q (e_i)) \otimes \bar{c} (N_0) \sigma, \hat{\nabla}_i \sigma \rangle
\\
= & | \hat{\nabla} \sigma |^2 + \tfrac{1}{4} |q|_g^2 | \sigma |^2
\\
& \quad - \tfrac{1}{2} \langle \nabla_i \sigma + \tfrac{1}{2}
(\varepsilon \otimes \bar{\varepsilon}) c (q (e_i)) \otimes \bar{c} (N_0)
\sigma, (\varepsilon \otimes \bar{\varepsilon}) c (q (e_i)) \otimes
\bar{c} (N_0) \sigma \rangle \\
& \quad - \tfrac{1}{2} \langle (\varepsilon \otimes \bar{\varepsilon}) c
(q (e_i)) \otimes \bar{c} (N_0) \sigma, \nabla_i \sigma + \tfrac{1}{2}
(\varepsilon \otimes \bar{\varepsilon}) c (q (e_i)) \otimes \bar{c} (N_0)
\sigma \rangle \\
= & | \hat{\nabla} \sigma |^2 - \tfrac{1}{4} |q|_g^2 | \sigma |^2
\\
& \quad - \tfrac{1}{2} \langle \nabla_i \sigma, (\varepsilon \otimes
\bar{\varepsilon}) c (q (e_i)) \otimes \bar{c} (N_0) \sigma \rangle -
\tfrac{1}{2} \langle (\varepsilon \otimes \bar{\varepsilon}) c (q (e_i))
\otimes \bar{c} (N_0) \sigma, \nabla_i \sigma \rangle .
\end{align}
  To collect all the calculations in the above,
\begin{align}
& \int_M | \widehat{D} \sigma |^2 \\
= & \int_M \left[ | \hat{\nabla} \sigma |^2 + \tfrac{1}{4} (R_g +
(\ensuremath{\operatorname{tr}}_g q)^2 - |q|^2_g) | \sigma |^2 \right] -
\int_{\partial M} (\langle c (e_n) D \sigma + \nabla_{e_n} \sigma, \sigma
\rangle) \\
& \quad - \tfrac{1}{2} \int_M [\langle \nabla_i \sigma, (\varepsilon
\otimes \bar{\varepsilon}) c (q (e_i)) \otimes \bar{c} (N_0) \sigma
\rangle + \langle (\varepsilon \otimes \bar{\varepsilon}) c (q (e_i))
\otimes \bar{c} (N_0) \sigma, \nabla_i \sigma \rangle] \\
& \quad + \int_M (\langle D \sigma, \Psi \sigma \rangle + \langle \Psi
\sigma, D \sigma \rangle) \\
= : & \int_M \left( | \hat{\nabla} \sigma |^2 + \tfrac{1}{4} (R_g +
(\ensuremath{\operatorname{tr}}_g q)^2 - |q|^2_g) | \sigma |^2 \right)
\\
& \quad - I_1 - \tfrac{1}{2} I_2 + I_3 . \label{Is}
\end{align}
  We start with $I_1$. We see by definitions \eqref{boundary connection} and
  \eqref{boundary Dirac} that
\begin{align}
& \langle c (e_n) D \sigma + \nabla_{e_n} \sigma, \sigma \rangle
\\
= & \langle - D^{\partial} \sigma -\mathcal{A} \sigma, \sigma \rangle
\\
= & - \tfrac{1}{4} \langle D^{\partial} ((1 + \chi) \sigma + (1 - \chi)
\sigma), ((1 + \chi) \sigma + (1 - \chi) \sigma) \rangle - \langle
\mathcal{A} \sigma, \sigma \rangle \\
= & - \tfrac{1}{4} \langle D^{\partial} (\sigma + \chi \sigma), \sigma -
\chi \sigma \rangle - \tfrac{1}{4} \langle D^{\partial} (\sigma - \chi
\sigma), \sigma + \chi \sigma \rangle - \langle \mathcal{A} \sigma, \sigma
\rangle
\end{align}
  where we have used the anti-commutative property $D^{\partial} \chi + \chi
  D^{\partial} = 0$.
  
  For $I_2$, it follows from
\begin{align}
& \nabla_i \langle \sigma, (\varepsilon \otimes \bar{\varepsilon}) c (q
(e_i)) \otimes \bar{c} (N_0) \sigma \rangle \\
= & \langle \nabla_i \sigma, (\varepsilon \otimes \bar{\varepsilon}) c (q
(e_i)) \otimes \bar{c} (N_0) \sigma \rangle + \langle \sigma, (\varepsilon
\otimes \bar{\varepsilon}) c (q (e_i)) \otimes \bar{c} (N_0) \nabla_i
\sigma \rangle \\
& \quad + \langle \sigma, (\varepsilon \otimes \bar{\varepsilon}) c
(\nabla_i q (e_i)) \otimes \bar{c} (N_0) \sigma \rangle,
\end{align}
  and integration by parts that
  \[ I_2 = \int_{\partial M} \langle \sigma, (\varepsilon \otimes
     \bar{\varepsilon}) c (q (e_n)) \bar{c} (N_0) \rangle - \int_M \langle
     \sigma, (\varepsilon \otimes \bar{\varepsilon}) c (\nabla_i q (e_i))
     \otimes \bar{c} (N_0) \sigma \rangle . \]
  The term $I_3$ remains, and we calculate as follows
\begin{align}
& \nabla_i \langle \Psi \sigma, c (e_i) \sigma \rangle \\
= & \langle \nabla_i \Psi \sigma, c (e_i) \sigma \rangle + \langle \Psi
\nabla_i \sigma, c (e_i) \sigma \rangle + \langle \Psi \sigma, c (e_i)
\nabla_i \sigma \rangle \\
= & \langle \nabla_i \Psi \sigma, c (e_i) \sigma \rangle + \langle \Psi D
\sigma, \sigma \rangle + \langle \Psi \sigma, D \sigma \rangle \\
= & \langle \nabla_i \Psi \sigma, c (e_i) \sigma \rangle + \langle D
\sigma, \Psi \sigma \rangle + \langle \Psi \sigma, D \sigma \rangle .
\end{align}
  By integration by parts,
\begin{align}
& I_3 \\
= & \int_{\partial M} \langle \Psi \sigma, c (e_n) \sigma \rangle - \int_M
\langle \nabla_i \Psi \sigma, c (e_i) \sigma \rangle \\
= & \int_{\partial M} \langle \Psi \sigma, c (e_n) \sigma \rangle - \int_M
\langle \sigma, (\varepsilon \otimes \bar{\varepsilon}) (c (\nabla
\ensuremath{\operatorname{tr}}_g q) \otimes \bar{c} (N_0)) \sigma \rangle
.
\end{align}
  Putting the calculations of $I_1$, $I_2$ and $I_3$ back to where they are
  defined \eqref{Is}, and the definitions of \eqref{mu} and \eqref{J} we
  finishing the proof of the proposition.
\end{proof}

We need to consider a section $\sigma$ of $S_{M_g} \otimes S_{M_{\delta}}$
with the local boundary condition $B$ \eqref{chi} in effect.

\begin{proposition}
  \label{sl with B}Let $\sigma$ be as in Proposition \ref{sl no B} with $\chi
  \sigma = \sigma$ and $\theta \in [0, \pi]$ such that $\cos \theta = \langle
  N_0, N \rangle$, then
\begin{align}
& \int_M | \hat{D} \sigma |^2 \\
\geqslant & \int_M \left( | \hat{\nabla} \sigma |^2 + \tfrac{1}{2} \langle
\sigma, \mu \sigma + (\varepsilon \otimes \bar{\varepsilon}) c
(J^{\sharp}) \otimes \bar{c} (N_0) \sigma \rangle \right) \\
& \quad + \tfrac{1}{2} \int_{\partial M} (H + \cos \theta
\ensuremath{\operatorname{tr}}_{\partial M} q - \sin \theta |q
(e_n)^{\top} | -\| \mathrm{d} N\|_{\ensuremath{\operatorname{tr}}}) |
\sigma |^2, \label{eq:SL with B}
\end{align}
  where $\| \mathrm{d} N\|_{\ensuremath{\operatorname{tr}}}$ is the trace norm
  of the map $\mathrm{d} N : T_{p_0} \partial M \to T_{N (p_0)} \mathbb{S}^{n
  - 1}$ and $p_0$ is any point on $\partial M$.
\end{proposition}

\begin{proof}
  We only have to deal with the two terms $\langle \mathcal{A} \sigma, \sigma
  \rangle$ and
  \begin{equation}
    \langle (\varepsilon \otimes \bar{\varepsilon}) c
    ((\ensuremath{\operatorname{tr}}_g q) e_n - q (e_n)) \bar{c} (N_0) \sigma,
    \sigma \rangle . \label{boundary remains}
  \end{equation}
  Let $Y$ be any Euclidean vector orthogonal to $N$ at $p_0$. We compute at
  $p_0$. Because both $(\varepsilon \otimes \bar{\varepsilon}) c (e_n) \bar{c}
  (Y)$ and $(\varepsilon \otimes \bar{\varepsilon}) c (e_n) \bar{c} (N)$ are
  self-adjoint and anti-commute, and $\chi \sigma = \sigma$, so
  \begin{equation}
    \langle (\varepsilon \otimes \bar{\varepsilon}) c (e_n) \bar{c} (Y)
    \sigma, \sigma \rangle = 0. \label{eq:vanishing n Y}
  \end{equation}
  Similarly,
  \begin{equation}
    \langle (\varepsilon \otimes \bar{\varepsilon}) c (v) \bar{c} (N) \sigma,
    \sigma \rangle = 0 \label{eq:vanishing v N}
  \end{equation}
  for any $v$ tangent to $\partial M$. We set $Y = N_0 - \langle N_0, N
  \rangle N$, as
  \[ (\ensuremath{\operatorname{tr}}_g q) e_n - q (e_n) =
     (\ensuremath{\operatorname{tr}}_{\partial M} q) e_n - q (e_n)^{\top}, \]
  where $q (e_n)^{\top}$ is the tangential component of $q (e_n)$ to $\partial
  M$, it follows from \eqref{eq:vanishing n Y} and \eqref{eq:vanishing v N}
  that
\begin{align}
& \langle (\varepsilon \otimes \bar{\varepsilon}) c
((\ensuremath{\operatorname{tr}}_g q) e_n - q (e_n)) \bar{c} (N_0) \sigma,
\sigma \rangle \\
= & \langle (\varepsilon \otimes \bar{\varepsilon}) c
((\ensuremath{\operatorname{tr}}_{\partial M} q) e_n - q (e_n)^{\top})
\bar{c} (\langle N_0, N \rangle N + Y) \sigma, \sigma \rangle \\
= & \ensuremath{\operatorname{tr}}_{\partial M} q \langle N_0, N \rangle
\langle (\varepsilon \otimes \bar{\varepsilon}) c (e_n) \bar{c} (N),
\sigma \rangle - \langle (\varepsilon \otimes \bar{\varepsilon}) c (q
(e_n)^{\top}) \bar{c} (Y) \sigma, \sigma \rangle \\
\geqslant & (H + \cos \theta \ensuremath{\operatorname{tr}}_{\partial M} q
- \sin \theta |q (e_n)^{\top} |) | \sigma |^2 .
\end{align}
  By definition of the trace norm $\| \mathrm{d}
  N\|_{\ensuremath{\operatorname{tr}}}$, there exists an orthonormal basis \
  $\{ \tilde{e}_j \}_{j = 1}^{n - 1}$ of the the subspace orthogonal to $N$
  and an orthonormal basis $\{e_i \}_{i = 1}^{n - 1}$ of $T_{x_0} \partial M$
  such that
  \[ \mathrm{d} N (e_i) = \bar{\nabla}_{e_i} N = \tau_i \tilde{e}_i, \]
  where $\tau_i \geqslant 0$ are the singular values of the map $\mathrm{d} N$
  and $\| \mathrm{d} N\|_{\ensuremath{\operatorname{tr}}} = \sum_{i = 1}^{n -
  1} \tau_i$. We see then using the definition \eqref{A} of $\mathcal{A}$ that
  \[ \langle \mathcal{A} \sigma, \sigma \rangle \geqslant \tfrac{1}{2} (H -\|
     \mathrm{d} N\|_{\ensuremath{\operatorname{tr}}}) | \sigma |^2 . \]
  And this concludes our proof of the proposition.
\end{proof}

\

\section{Initial data set rigidity for polyhedra}\label{sec:ids rigidity for
polyhedra}

We dedicate this section to the proof of Theorem \ref{dihedral ids}. First, we
briefly recall Brendle's smoothing construction {\cite{brendle-scalar-2024}}.
We then find a solution to a Dirac type equation \eqref{Dirac type} on the
polyhedral initial data set $(\Omega, g, q)$. The rest of this section is
occupied by the rigidity analysis of \eqref{Dirac type}.

\subsection{Brendle smoothing}

For a sufficiently large $\lambda > 0$, Brendle's smoothing of the polyhedron
$\Omega$ is given by
\begin{equation}
  \Omega_{\lambda} = \left\{ \sum_{\ell} e^{\lambda u_{\ell}} \leqslant 1
  \right\} . \label{Omega approx}
\end{equation}
Let $N_{\lambda} : \partial \Omega_{\lambda} \to \mathbb{S}^{n - 1}$ be given
by
\[ N_{\lambda} = \sum_{\ell \in \Lambda} e^{\lambda u_{\ell}} | \nabla
   u_{\ell} | N_{\ell} \left| \sum_{\ell \in \Lambda} e^{\lambda u_{\ell}} |
   \nabla u_{\ell} |N_{\ell} \right|^{- 1} . \]
It is clear that by deforming the metric $g$ to the flat metric that
$N_{\lambda}$ is homotopic to the Euclidean Gauss map of $\partial
\Omega_{\lambda}$. We define the bundle map $\chi_{\lambda}$ by
\[ \chi_{\lambda} \sigma = (\varepsilon \otimes \bar{\varepsilon}) c
   (\nu_{\lambda}) \bar{c} (N_{\lambda}) \sigma \]
where $\nu_{\lambda}$ is the unit outward normal to $\partial
\Omega_{\lambda}$.

Using the index theory from {\cite{wang-dihedral-2023}}, we can find a
solution $\sigma^{(\lambda)}$ to the following
\[ \hat{D} \sigma^{(\lambda)} = 0 \text{ in } \Omega_{\lambda}, \text{ }
   \chi_{\lambda} \sigma^{(\lambda)} = \sigma^{(\lambda)} \text{ along }
   \partial \Omega_{\lambda} . \]
Replacing $H$ by $H + \cos \theta \ensuremath{\operatorname{tr}}_{\partial M}
q - \sin \theta |q (e_n)^{\top} |$ in the same argument of
{\cite{brendle-scalar-2024}} , we can show that there exists a subsequence
$\{\lambda_j \}_{j \in \mathbb{N}}$ such that $\sigma^{(\lambda_l)}$ converges
in $C^{\infty}_{\ensuremath{\operatorname{loc}}}$ away from the edges of
$\Omega$ to a non-zero section $\sigma$, $N_{\lambda_l}$ converge to
$N_{\ell}$ on each face $F_{\ell}$, and
\begin{equation}
  \hat{D} \sigma = 0 \text{ in } \Omega, \text{ } \chi_{\ell} \sigma = \sigma
  \text{ along } F_{\ell}, \label{Dirac type}
\end{equation}
for all $\ell \in \Lambda$, where $\chi_{\ell} = (\varepsilon \otimes
\bar{\varepsilon}) c (\nu_{\ell}) \bar{c} (N_{\ell}) \sigma$.

The crucial inequality \eqref{eq:SL with B} is also preserved in the limit
(the right hand side is uniformly bounded below), so using the dominant energy
condition \eqref{eq:dec}, tilted dominant energy condition \eqref{tilt dec}
along every face $F_{\ell}$ of $\Omega$ (assumptions in Theorem \ref{dihedral
ids}), we conclude that
\[ \hat{\nabla}_i \sigma = \nabla_i \sigma + \tfrac{1}{2} (\varepsilon \otimes
   \varepsilon) c (q (e_i)) \otimes \bar{c} (N_0) \sigma = 0 \text{ in }
   \Omega, \text{ } \chi_{\ell} \sigma = \sigma \text{ along } F_{\ell}
   \label{killing} \]
for all $\ell \in \Lambda$.

\

We find it convenient to work with $\sigma$ in components. We fix an
orthonormal basis $\{ \bar{s}_{\alpha} \}_{\alpha = 1}^m$ of
$S_{\Omega_{\delta}}$, $m : = 2^{n / 2}$, then $\sigma = \sum_{\alpha}
s_{\alpha} \otimes \bar{s}_{\alpha}$ for some sections $s_{\alpha}$, $\alpha =
1, \cdots, m$ of the spinor bundle $S_{\Omega_g}$. We denote by $s$ the
$m$-tuple of spinors $(s_1, \ldots, s_m)$. The $m$-tuple of spinors $s$ should
be understood as a spinor-valued column vector. Given a Euclidean vector $X$,
we define the matrix $\omega_X$ by
\begin{equation}
  (\omega_X s)_{\alpha} = \sum_{\beta = 1}^m \omega_{X \alpha \beta} s_{\beta}
  : = \sum_{\beta = 1}^m \langle \bar{\varepsilon} \bar{c} (X)
  \bar{s}_{\beta}, \bar{s}_{\alpha} \rangle s_{\beta} . \label{omega X}
\end{equation}
It is easy to check that $\omega_X$ is a Hermitian matrix from the property of
$\bar{\varepsilon} \bar{c} (X)$. Hence the $m$-tuple of spinors $s$ satisfies
\begin{equation}
  \nabla_i s + \tfrac{1}{2} \varepsilon \cdot q (e_i) \cdot \omega_{N_0} s = 0
  \text{ in } \Omega \label{eq:killing}
\end{equation}
and
\begin{equation}
  \varepsilon \cdot \nu_{\ell} \cdot \omega_{N_{\ell}} s = s \text{ along }
  F_{\ell} \label{eq:bdry condition}
\end{equation}
for all $\ell \in \Lambda$. From here on we also use the dot $\cdot$ to denote
the Clifford multiplication for brevity.

\subsection{Linear independence of components of $s$}

We introduce the \text{{\itshape{formal inner product}}} $\langle c, s
\rangle$ of a spinor $c \in \mathbb{C}^m$ and an $m$-tuple $s$ of spinors by
\[ \langle c, s \rangle = \sum_{\alpha = 1}^m \bar{c}_{\alpha} s_{\alpha} \]
where $\bar{c}_{\alpha}$ is the complex conjugate of $c_{\alpha} \in
\mathbb{C}$. To avoid double angular brackets, we use $\langle c_1, s_1
\rangle \cdot \langle c_2, s_2 \rangle$ to denote $\langle \langle c_1, s_1
\rangle, \langle c_2, s_2 \rangle \rangle$ when there is no confusion. We can
easily check that
\[ \langle c, \omega_X s \rangle = \langle \omega_X c, s \rangle \]
for any $c \in \mathbb{C}^m$ and $m$-tuple of spinors $s$ from that $\omega_X$
is Hermitian.

\begin{lemma}
  \label{lm:highest point}There exists a point $x_0 \in \partial \Omega$ and a
  unit vector $\xi$ at $x_0$ such that $\omega_{N_0} s = \varepsilon \cdot \xi
  \cdot s$ at $x_0$.
\end{lemma}

\begin{proof}
  Let $x_0$ be a point with largest $x^1$ coordinate among all points in
  $\Omega$, it is clear that $x_0$ exists since $\Omega$ is compact and $x_0
  \in \partial \Omega$. Let $L$ be the hyperplane orthogonal to $N_0$ through
  $x_0$.
  
  If $\partial \Omega \cap L$ contains only $x_0$, meaning that $x_0$ is a
  vertex of $\partial \Omega$, then the linear span of the set $\{N_{\ell} :
  \text{ } x_0 \in F_{\ell} \}$ of unit normals is the full $\mathbb{R}^n$.
  Denote $\Lambda' = \{\ell \in \Lambda : \text{ } x_0 \in F_{\ell} \}$. So
  there exist numbers $a_{\ell}$ such that $N_0 = \sum_{\ell \in \Lambda'}
  a_{\ell} N_{\ell}$ and
  \[ \omega_{N_0} s = \sum_{\ell \in \Lambda'} a_{\ell} \omega_{N_{\ell}} s =
     \sum_{\ell \in \Lambda'} a_{\ell} \varepsilon \cdot \nu_{\ell} \cdot s =
     \varepsilon \cdot (\sum_{\ell \in \Lambda'} a_{\ell} \nu_{\ell}) \cdot s
  \]
  at $x_0$. Let $\xi = \sum_{\ell \in \Lambda'} a_{\ell} \nu_{\ell}$. It
  remains to check that $\xi$ is of unit length at $x_0$. Indeed,
  \[ |s|^2 = \langle \omega_{N_0} s, \omega_{N_0} s \rangle = \langle
     \varepsilon \cdot \xi \cdot s, \varepsilon \cdot \xi \cdot s \rangle =
     |s|^2 | \xi |^2 . \]
  If $\partial \Omega \cap L$ contains more than $x_0$, then the tangent space
  of $\partial \Omega$ at $x_0$ splits off an $\mathbb{R}^k$ factor ($1
  \leqslant k \leqslant n - 1$) and a factor $C$ which is a cone with a
  vertex. It is clear that $N_0$ is orthogonal to the $\mathbb{R}^k$ factor
  since $x_0$ is a point with largest $x_1$ coordinate. And if $F_{\ell}$ is a
  face such that $x_0 \in F_{\ell}$, then the $\mathbb{R}^k$ factor also lies
  in the tangent space of $F_{\ell}$ at $x_0$. So $N_{\ell}$ is orthogonal to
  the $\mathbb{R}^k$ factor and we are reduced to the first case where
  $\partial \Omega \cap L$ contains only $x_0$.
\end{proof}

From now on we fix such a point $x_0$ as in Lemma \ref{lm:highest point}.

\begin{lemma}
  Let $F$ be a face of $\Omega$ and N be its Euclidean unit normal. Then for
  any $c_i \in \mathbb{C}^m$, $i = 1, 2$ that
  \begin{equation}
    \langle c_1, (1 + \omega_N) s \rangle \cdot \langle c_2, (1 - \omega_N) s
    \rangle = 0 \label{eq:pre orth}
  \end{equation}
  along $F$.
\end{lemma}

\begin{proof}
  The proof is a direct calculation by using that $\omega_N s = \varepsilon
  \cdot \nu \cdot s$, $(\varepsilon \cdot \nu)^2 = 1$ and that $\varepsilon
  \cdot \nu \cdot$ is self-adjoint.
\end{proof}

\begin{remark}
  \label{ortho at x0}By Lemma \ref{lm:highest point}, the same proof works at
  $x_0$ and it leads to
  \[ \langle c_1, (1 + \omega_{N_0}) s \rangle \cdot \langle c_2, (1 -
     \omega_{N_0}) s \rangle = 0 \]
  for any $c_i \in \mathbb{C}^m$ at $x_0$.
\end{remark}

\begin{proposition}
  The components of $s$ are linearly independent.
\end{proposition}

\begin{proof}
  To this end, we introduce the following sets
  \[ L_{\pm} = \{c \in \mathbb{C}^m : \text{ } \langle c, (1 \pm \omega_{N_0})
     s \rangle = 0 \text{ everywhere in } \Omega\} . \]
  We show that $L_+ \cap L_- = \{0\}$. First, we note that
\begin{align}
& \nabla_i \langle c, (1 \pm \omega_{N_0}) s \rangle \\
= & \tfrac{1}{2} \langle c, (1 \pm \omega_{N_0}) \varepsilon \cdot q (e_i)
\cdot \omega_{N_0} s \rangle \\
= & \tfrac{1}{2} \varepsilon \cdot q (e_i) \cdot \langle c, (\omega_{N_0}
\pm 1) s \rangle .
\end{align}
  It follows from an ODE argument that if $\langle c, (1 \pm \omega_{N_0}) s
  \rangle$ vanishes at some point of $\Omega$, then it vanishes on all of
  $\Omega$. Let $p_0$ be a boundary point of $\partial \Omega$ and $N$ is the
  unit normal of $\partial \Omega$ at $p_0$ pointing outward of $\Omega$. For
  any $c \in L_+ \cap L_-$, we have
\begin{align}
& \langle \omega_N c, (1 \pm \omega_{N_0}) s \rangle \\
= & \langle c, \omega_N (1 \pm \omega_{N_0}) s \rangle \\
= & \langle c, (\omega_N \pm (2 \langle N, N_0 \rangle - \omega_{N_0}
\omega_N)) s \rangle \\
= & \langle c, (1 \mp \omega_{N_0}) \omega_N s \rangle \pm \langle N, N_0
\rangle \langle c, s \rangle \\
= & \langle c, (1 \mp \omega_{N_0}) \omega_N s \rangle
\end{align}
  where we have used that $\langle c, s \rangle = 0$. Using the boundary
  condition \eqref{eq:bdry condition}, $\omega_N s = \varepsilon \cdot \nu
  \cdot s$. So at $p_0$,
  \[ \langle \omega_N c, (1 \pm \omega_{N_0}) s \rangle = \langle c, (1 \mp
     \omega_{N_0}) \varepsilon \cdot \nu \cdot s \rangle = \varepsilon \cdot
     \nu \cdot \langle c, (1 \mp \omega_{N_0}) s \rangle = 0. \]
  The same ODE argument shows that $\langle \omega_N c, (1 \pm \omega_{N_0}) s
  \rangle = 0$ on all of $\Omega$. So $\omega_N c \in L_+ \cap L_-$. Since the
  linear span of the unit normals of $\partial \Omega$ is the whole Euclidean
  space and $\ensuremath{\operatorname{End}} (\mathbb{C}^m)$ is generated by
  all $\omega_N$, we conclude that $L_+ \cap L_-$ is invariant under all
  $\ensuremath{\operatorname{End}} (\mathbb{C}^m)$. Hence $L_+ \cap L_-$ is
  either $\{0\}$ or $\mathbb{C}^m$. But the latter is not possible since $s$
  has at least one non-zero component.
  
  We now count the dimensions of $L_+$ and $L_-$. Because $L_{\pm}$ obviously
  contains the linear space $\{(1 \mp \omega_{N_0}) c : \text{ } c \in
  \mathbb{C}^m \}$ which is of dimension $m / 2$, and $L_+ \cap L_- = \{0\}$,
  so
  \[ L_{\pm} = \{(1 \mp \omega_{N_0}) c : \text{ } c \in \mathbb{C}^m \}, \]
  and it implies that both $(1 + \omega_{N_0}) s$ and $(1 - \omega_{N_0}) s$
  have respectively $m / 2$ linear independent components.
  
  For any $c_1, c_2 \in \mathbb{C}^m$, set $\psi_+ = \langle c_1, (1 +
  \omega_{N_0}) s \rangle$ and $\psi_- = \langle c_2, (1 - \omega_{N_0}) s
  \rangle$. It suffices to show that $\langle \psi_+, \psi_- \rangle = 0$ on
  $\Omega$ to finish the proof. From \eqref{eq:killing}, we deduce that
  $\nabla_i \psi_{\pm} \pm \tfrac{1}{2} \varepsilon \cdot q (e_i) \cdot
  \psi_{\pm} = 0$ and $\nabla_i \langle \psi_+, \psi_- \rangle = 0$. So
  $\langle \psi_+, \psi_- \rangle$ is a constant. By Remark \ref{ortho at x0},
  $\langle \psi_+, \psi_- \rangle = 0$ at the $x_0 \in \partial \Omega$. So,
  \begin{equation}
    0 = \langle \psi_+, \psi_- \rangle = \langle c_1, (1 + \omega_{N_0}) s
    \rangle \cdot \langle c_2, (1 - \omega_{N_0}) s \rangle \label{cross
    orthogonal}
  \end{equation}
  on all $\Omega$ implying that components of $(1 + \omega_{N_0}) s$ and $(1 -
  \omega_{N_0}) s$ are orthogonal. Therefore, the components of $s$ are
  linearly independent.
\end{proof}

\subsection{Foliation by capillary MOTS}\label{foliation by MOTS}

We adapt the analysis due to Beig-Chrusciel {\cite{beig-killing-1996}}. From
here on, $s$ is always the $m$-tuple of spinors satisfying \eqref{eq:killing}
and \eqref{eq:bdry condition}.

\begin{lemma}
  \label{lm: tangent normal product of spinors}Let $\psi$ be a non-zero spinor
  and $\xi$ be a unit tangent vector at a point $p \in M$ such that
  \[ \langle \varepsilon \cdot \xi \cdot \psi, \psi \rangle = \max_{e \in T_p
     M, |e| = 1} \langle \varepsilon \cdot e \cdot \psi, \psi \rangle . \]
  Then for any tangent vector $e \in T_p M$ orthogonal to $\xi$, $\langle
  \varepsilon \cdot e \cdot \psi, \psi \rangle = 0$.
\end{lemma}

\begin{proof}
  For any tangent vector $e \in T_p M$ orthogonal to $\xi$, we can fix a
  smooth curve $e (t) : [- 1, 1] \to T_p M$ with $|e (t) | = 1$, passing
  through $e (0) = \xi$ and $e' (0) = e$, then
  \[ v (t) : = \langle \varepsilon \cdot e (t) \cdot \psi, \psi \rangle \]
  is a function attaining its maximum at $t = 0$ and hence
  \[ v' (0) = \langle \varepsilon \cdot e' (0) \cdot \psi, \psi \rangle =
     \langle \varepsilon \cdot e \cdot \psi, \psi \rangle = 0. \]
  
\end{proof}

\begin{remark}
  \label{rk:f geq W}By compactness of unit tangent vectors at $p$, $\max_{|e|
  = 1} \langle \varepsilon \cdot e \cdot \psi, \psi \rangle$ is always
  achieved by some unit tangent vector $\xi$. So
  \[ \sum_i \langle \varepsilon \cdot e_i \cdot \psi, \psi \rangle^2 = \langle
     \varepsilon \cdot \xi \cdot \psi, \psi \rangle^2 \leqslant | \psi |^4 .
  \]
  If $\langle \varepsilon \cdot \xi \cdot \psi, \psi \rangle = | \psi |^2$,
  then $\varepsilon \cdot \xi \cdot \psi = \psi$.
\end{remark}

\begin{lemma}
  \label{lm:one spinor omega N0 identification}Let $\omega_{N_0} c = c$, we
  set $\psi = \langle c, s \rangle$, $f = | \psi |^2$ and
  \[ W = W^j e_j = \langle \varepsilon \cdot e_j \cdot \psi, \psi \rangle e_j
     . \]
  Then $\varepsilon \cdot W \cdot \psi = f \psi$.
\end{lemma}

\begin{proof}
  We choose a geodesic orthonormal frame $\{e_i \}$ at a point $p \in \Omega$,
  that is, $\nabla_i e_j$ vanishes at $p$. Then at the point $p$,
  \begin{equation}
    \nabla_i f = - \langle \varepsilon \cdot q (e_i) \cdot \psi, \psi \rangle
    = - q_{i j} W^j, \label{f derivative}
  \end{equation}
  and
\begin{align}
\nabla_i W = & \nabla_i (\langle \varepsilon \cdot e_j \cdot \psi, \psi
\rangle e_j) \\
= & \langle \varepsilon \cdot e_j \cdot \nabla_i \psi, \psi \rangle e_j +
\langle \varepsilon \cdot e_j \cdot \psi, \nabla_i \psi \rangle e_j
\\
= & - \tfrac{1}{2} \langle \varepsilon \cdot e_j \cdot \varepsilon \cdot q
(e_i) \cdot \psi, \psi \rangle e_j - \tfrac{1}{2} \langle \varepsilon
\cdot e_j \cdot \psi, \varepsilon \cdot q (e_i) \cdot \psi \rangle e_j
\\
= & - \tfrac{1}{2} \langle (\varepsilon \cdot e_j \cdot \varepsilon \cdot
q (e_i) + \varepsilon \cdot q (e_i) \cdot \varepsilon \cdot e_j) \cdot
\psi, \psi \rangle e_j \\
= & - \langle e_j, q (e_i) \rangle \langle \psi, \psi \rangle e_j
\\
= & = - q_{i j} e_j | \psi |^2 .
\end{align}
  That is,
  \begin{equation}
    \nabla_i W = - f q_{i j} e_j = - f q (e_i) . \label{grad of W}
  \end{equation}
  It easily follows from \eqref{f derivative} and \eqref{grad of W} that
  $\mathrm{d} (f^2 - |W|^2) = 0$, hence $f^2 - |W|^2$ is a constant. By Lemma
  \ref{lm:highest point}, there exists a unit vector $\xi$ such that
  $\varepsilon \cdot \xi \cdot \omega_{N_0} s = s$ at $x_0$, so $\varepsilon
  \cdot \xi \cdot \psi = \psi$. By Remark \ref{rk:f geq W}, $f = |W|$ at $x_0$
  and hence $f = |W|$ in $\Omega$.
  
  We finish the proof by the following computation,
\begin{align}
0 = f^2 - |W|^2 = & \langle \psi, \psi \rangle^2 - \sum_i \langle
\varepsilon \cdot e_i \cdot \psi, \psi \rangle^2 \\
= & \langle \psi, \psi \rangle^2 - \sum_i W^i \langle \varepsilon \cdot
e_i \cdot \psi, \psi \rangle \\
= & \langle \psi, \psi \rangle^2 - \langle \varepsilon \cdot W \cdot \psi,
\psi \rangle \\
= & f (\langle \psi, \psi \rangle - \langle \varepsilon \cdot \tfrac{W}{f}
\cdot \psi, \psi \rangle) .
\end{align}
  Therefore, $\varepsilon \cdot W \cdot \psi = f \psi$ again by Remark
  \ref{rk:f geq W}.
\end{proof}

\begin{remark}
  It is easy to see that $\xi = W / f$ at $x_0$.
\end{remark}

\begin{lemma}
  \label{lm:W is gradient}There exists a function $t$ such that $W = \nabla
  t$.
\end{lemma}

\begin{proof}
  We choose a geodesic orthonormal frame $\{e_i \}$ at a point $p \in \Omega$,
  that is, $\nabla_i e_j$ vanishes at $p$. Without loss of generality, we do
  all the calculations at $p$. Let $\eta$ be the dual 1-form of $W$, that is,
  \[ \eta (e_i) = \langle \varepsilon \cdot e_i \cdot \psi, \psi \rangle = W^i
     . \]
  The following
\begin{align}
& \mathrm{d} \eta (e_i, e_j) \\
= & e_i (\eta (e_j)) - e_j (\eta (e_i)) \\
= & e_i \langle W, e_j \rangle - e_j \langle W, e_i \rangle \\
= & \langle \nabla_i W, e_j \rangle - \langle \nabla_j W, e_i \rangle
\\
= & \langle - f q (e_i), e_j \rangle - \langle - f q (e_j), e_i \rangle =
0,
\end{align}
  shows that $\eta$ is a closed 1-form. By Poincar{\'e} lemma, there exists a
  function $t$ such that $\eta = \mathrm{d} t$ and $W$ defines a foliation in
  $(\Omega, g)$, so $W = \nabla t$.
\end{proof}

\begin{lemma}
  \label{lm:foliation with parallel spinor}Let $t$ be the function obtained in
  Lemma \ref{lm:W is gradient} and $\Sigma_t$ be a level set of the function
  $t$, then $\Sigma_t$ is a marginally outer trapped surface with at least one
  parallel spinor.
\end{lemma}

\begin{proof}
  Note that $W / f$ is a unit normal to $\Sigma_t$ since $W = \nabla t$. We
  compute the second fundamental form of $\Sigma_t$. Since $W / f$ be a unit
  vector by Lemma \ref{lm:one spinor omega N0 identification}, we choose an
  orthonormal frame $\{e_i \}$ such that $W / f = e_n$, so for $i, j < n$,
\begin{align}
& \langle \nabla_i e_n, e_j \rangle \\
= & \langle \tfrac{f \nabla_i W}{f^2} - \tfrac{W \nabla_i f}{f^2}, e_j
\rangle \\
= & f^{- 1} \langle \nabla_i W, e_j \rangle \\
= & \langle - q (e_i), e_j \rangle = - q_{i j} .
\end{align}
  This says that each $\Sigma_t$ is a MOTS (that is, $\sum_{i = 1}^{n - 1}
  (\langle \nabla_{e_i} e_n, e_i \rangle + q_{i i}) = 0$). In fact, its null
  second fundamental form vanishes. Also, by Lemma \ref{lm:one spinor omega N0
  identification}, $\varepsilon \cdot e_n \cdot \psi = \psi$ and so $e_n \cdot
  \psi = \varepsilon \cdot \psi$. Recall that the hypersurface spinor
  connection on $\Sigma_t$ is defined by
  \[ \nabla_i^{\partial} \psi = \nabla_i \psi + \tfrac{1}{2} (\nabla_i e_n)
     \cdot e_n \cdot \psi . \]
  So by \eqref{eq:killing} and that $\langle \nabla_i e_n, e_j \rangle = -
  q_{i j}$ for $i, j < n$ along $\Sigma_t$,
\begin{align}
& \nabla_i^{\partial} \psi \\
= & - \tfrac{1}{2} q_{i j} \varepsilon \cdot e_j \cdot \psi - \sum_{j \neq
n} \tfrac{1}{2} q_{i j} e_j \cdot e_n \cdot \psi \\
= & - \tfrac{1}{2} q_{i j} \varepsilon \cdot e_j \cdot \psi - \sum_{j \neq
n} \tfrac{1}{2} q_{i j} e_j \cdot \varepsilon \cdot \psi . \\
= & - \tfrac{1}{2} q_{i n} \varepsilon \cdot e_n \cdot \psi = -
\tfrac{1}{2} q_{i n} \psi .
\end{align}
  Therefore,
\begin{align}
& \nabla_i^{\partial} (\psi | \psi |^{- 1}) \\
= & \nabla_i^{\partial} \psi | \psi |^{- 1} - \psi \nabla_i^{\partial} (|
\psi |^2)^{- 1 / 2} \\
= & \nabla_i^{\partial} \psi | \psi |^{- 1} - \psi | \psi |^{- 3} \langle
\nabla_i^{\partial} \psi, \psi \rangle \\
= & - \tfrac{1}{2} q_{i n} \psi | \psi |^{- 1} + \tfrac{1}{2} \psi | \psi
|^{- 3} q_{i n} \langle \psi, \psi \rangle = 0.
\end{align}
  We see that $\psi | \psi |^{- 1}$ is parallel with respect to the connection
  $\nabla^{\partial}$. That is, $\Sigma_t$ carries a parallel spinor $\psi |
  \psi |^{- 1}$.
\end{proof}

\subsection{Capillarity of the leaf}

\begin{lemma}
  \label{lm:capillary}Each $\Sigma_t$ intersects the face $F$ with angle
  $\angle_g (\Sigma_t, F)$ is equal to the angle formed by the two Euclidean
  vectors $N$ and $N_0$.
\end{lemma}

\begin{proof}
  Let $\nu$ be the unit outward unit normal to the face $F$. We only have to
  compute $\langle W, \nu \rangle$. We have
  \[ \langle W, \nu \rangle = \sum_i \nu_i \langle \varepsilon \cdot e_i \cdot
     \psi, \psi \rangle = \langle \varepsilon \cdot \nu \cdot \psi, \psi
     \rangle = \langle c, \varepsilon \cdot \nu \cdot s \rangle \cdot \langle
     c, s \rangle . \]
  Using the boundary condition \eqref{eq:bdry condition}, $\langle W, \nu
  \rangle = \langle c, \omega_N s \rangle \cdot \langle c, s \rangle$, then we
  see
\begin{align}
& \langle W, \nu \rangle \\
= & \langle N, N_0 \rangle \langle c, \omega_{N_0} s \rangle \cdot \langle
c, s \rangle + \langle c, \omega_{N - \langle N_0, N \rangle N_0} s
\rangle \cdot \langle c, s \rangle \\
= & \langle N, N_0 \rangle  | \psi |^2 = \langle N, N_0 \rangle f.
\end{align}
  The fact that $\langle c, \omega_{N - \langle N_0, N \rangle N_0} s \rangle
  \cdot \langle c, s \rangle = 0$ follows from \eqref{cross orthogonal}. In
  the last line, we have used $\omega_{N_0} c = c$. So the angle $\cos
  \angle_g (W, \nu) = \langle N, N_0 \rangle$.
\end{proof}

\subsection{Flatness of the leaf}

We have shown that $\Sigma_t$ carries at least one parallel spinor, in order
to show that $\Sigma_t$ is flat, we need to show that $\psi = \langle c, s
\rangle$ determine the same $\Sigma_t$ for any $c \in \mathbb{C}^m$. To this
end, we shall generalize the analysis of Beig-Chrusciel to the settings of a
twisted spinor bundle ($m$-tuple of spinors).

Let $\Lambda_{\pm} = \{c \in \mathbb{C}^m : \text{ } \omega_{N_0} c = \pm
c\}$. We have the following lemma.

\begin{lemma}
  \label{lm:commute omega X}Let $F$ be a face of $\Omega$, $N$ be its
  Euclidean outward unit normal and $X$ be a unit Euclidean vector normal to
  $N_0$ such that $N = \langle N, X \rangle X + \langle N_0, N \rangle N_0$,
  then
  \begin{equation}
    \langle c_1, s \rangle \cdot \langle c_2, \omega_X s \rangle = \langle
    c_1, \omega_X s \rangle \cdot \langle c_2, s \rangle \label{commute with
    omega X}
  \end{equation}
  for any $c_1 \in \Lambda_+$ and $c_2 \in \Lambda_-$.
\end{lemma}

\begin{proof}
  We set $a = \langle N, N_0 \rangle$ and $b = \langle N, X \rangle$. Without
  loss of generality, we assume that $a \neq \pm 1$. Then $N = a N_0 + b X$
  and $a^2 + b^2 = 1$. From \eqref{eq:pre orth}, $c_1 \in \Lambda_+$ and $c_2
  \in \Lambda_-$, we see
\begin{align}
0 = & \langle c_1, (1 + \omega_N) s \rangle \cdot \langle c_2, (1 -
\omega_N) s \rangle \\
= & \langle c_1, (1 + a) s + b \omega_X s \rangle \cdot \langle c_2, (1 +
a) s - b \omega_X s \rangle .
\end{align}
  Simple expanding yields
\begin{align}
0 & = (1 + a)^2 \langle c_1, s \rangle \cdot \langle c_2, s \rangle - b^2
\langle c_1, \omega_X s \rangle \cdot \langle c_2, \omega_X s \rangle
\\
& \qquad - b (1 + a) [\langle c_1, s \rangle \cdot \langle c_2, \omega_X
s \rangle - \langle c_1, \omega_X s \rangle \cdot \langle c_2, s \rangle]
.
\end{align}
  From \eqref{cross orthogonal}, $\langle c_1, s \rangle \cdot \langle c_2, s
  \rangle = 0$. Because $\omega_{N_0} c_1 = c_1$, and $X$ and $N_0$ are
  orthogonal, so we have
  \[ \omega_{N_0} (\omega_X c_1) = - \omega_X \omega_{N_0} c_1 = - \omega_X
     c_1 \in \Lambda_- . \]
  Similarly, $\omega_X c_2 \in \Lambda_+$. By \eqref{cross orthogonal} again,
  $\langle c_1, \omega_X s \rangle \cdot \langle c_2, \omega_X s \rangle = 0$.
  The lemma is proved.
\end{proof}

\begin{lemma}
  \label{communte with omega X plus minus}Let $c_1 \in \Lambda_+$ and $c_2 \in
  \Lambda_-$. Then
  \[ \langle \omega_X c_1, s \rangle \cdot \langle c_2, s \rangle - \langle
     c_1, s \rangle \cdot \langle \omega_X c_2, s \rangle = 0 \]
  for all $X$ in $\mathbb{R}^n$.
\end{lemma}

\begin{proof}
  Let $N$ be a normal to a face $F$ of $\Omega$, $X$ be a unit vector tangent
  to $F$ such that $N = a N_0 + b X$. We set
  \[ z = \langle \omega_X c_1, s \rangle \cdot \langle c_2, s \rangle -
     \langle c_1, s \rangle \cdot \langle \omega_X c_2, s \rangle \]
  and the vector field $Z = Z_i e_i$ by
  \[ Z_i = \langle \omega_X c_1, \varepsilon \cdot e_i \cdot s \rangle \cdot
     \langle c_2, s \rangle + \langle c_1, \varepsilon \cdot e_i \cdot s
     \rangle \cdot \langle \omega_X c_2, s \rangle . \]
  Note that $z$ and $Z$ are complex-valued. By Lemma \ref{lm:highest point},
  $z = - \langle Z, \xi \rangle$ at $x_0$. Assume that $e_i$ is a vector
  orthogonal to $\xi$ at $x_0$, then again by Lemma \ref{lm:highest point} at
  $x_0$,
\begin{align}
& \langle \omega_X c_1, \varepsilon \cdot e_i \cdot s \rangle \cdot
\langle c_2, s \rangle \\
= & \langle - \omega_{N_0} \omega_X c_1, \varepsilon \cdot e_i \cdot s
\rangle \cdot \langle c_2, s \rangle \\
= & - \langle \omega_X c_1, \omega_{N_0} \varepsilon \cdot e_i \cdot s
\rangle \cdot \langle c_2, s \rangle \\
= & - \langle \omega_X c_1, \varepsilon \cdot e_i \cdot \omega_{N_0} s
\rangle \cdot \langle c_2, s \rangle \\
= & - \langle \omega_X c_1, \varepsilon \cdot e_i \cdot \varepsilon \cdot
\xi \cdot s \rangle \cdot \langle c_2, s \rangle .
\end{align}
  And also
\begin{align}
& \langle \omega_X c_1, \varepsilon \cdot e_i \cdot s \rangle \cdot
\langle c_2, s \rangle \\
= & \langle \omega_X c_1, \varepsilon \cdot e_i \cdot s \rangle \cdot
\langle - \omega_{N_0} c_2, s \rangle \\
= & - \langle \omega_X c_1, \varepsilon \cdot e_i \cdot s \rangle \cdot
\langle c_2, \omega_{N_0} s \rangle \\
= & - \langle \omega_X c_1, \sqrt{- 1} e_i \cdot s \rangle \cdot \langle
c_2, \varepsilon \cdot \xi \cdot s \rangle . \\
= & - \langle \omega_X c_1, \varepsilon \cdot \xi \cdot \varepsilon \cdot
e_i \cdot s \rangle \cdot \langle c_2, s \rangle .
\end{align}
  Therefore,
\begin{align}
& \langle \omega_X c_1, \varepsilon \cdot e_i \cdot s \rangle \cdot
\langle c_2, s \rangle \\
= & - \tfrac{1}{2} (\langle \omega_X c_1, \varepsilon \cdot \xi \cdot
\varepsilon \cdot e_i \cdot s \rangle \cdot \langle c_2, s \rangle +
\langle \omega_X c_1, \varepsilon \cdot e_i \cdot \varepsilon \cdot \xi
\cdot s \rangle \cdot \langle c_2, s \rangle) \\
= & - \tfrac{1}{2} \langle \omega_X c_1, (\varepsilon \cdot \xi \cdot
\varepsilon \cdot e_i + \varepsilon \cdot e_i \cdot \varepsilon \cdot \xi)
\cdot s \rangle \cdot \langle c_2, s \rangle \\
= & 0.
\end{align}
  Similarly, $\langle c_1, \varepsilon \cdot e_i \cdot s \rangle \cdot \langle
  \omega_X c_2, s \rangle = 0$ at $x_0$. So we conclude that $|z| = |Z|$ at
  $x_0$. We calculate $\nabla_i z$ as follows. We have by \eqref{eq:killing}
  that
\begin{align}
& \nabla_i z \\
= & \tfrac{1}{2} \langle \omega_X c_1, \varepsilon \cdot q (e_i) \cdot s
\rangle \cdot \langle c_2, s \rangle + \tfrac{1}{2} \langle \omega_X c_1,
s \rangle \cdot \langle c_2, \varepsilon \cdot q (e_i) \cdot s \rangle
\\
& \quad + \tfrac{1}{2} \langle c_1, \varepsilon \cdot q (e_i) \cdot s
\rangle \cdot \langle \omega_X c_2, s \rangle + \tfrac{1}{2} \langle c_1,
s \rangle \cdot \langle \omega_X c_2, \varepsilon \cdot q (e_i) \cdot s
\rangle \\
= & \langle \omega_X c_1, \varepsilon \cdot q (e_i) \cdot s \rangle \cdot
\langle c_2, s \rangle \\
& \quad + \langle c_1, \varepsilon \cdot q (e_i) \cdot s \rangle \cdot
\langle \omega_X c_2, s \rangle \\
= & q_{i j} Z_j .
\end{align}
  Also,
  \[ Z_i = \langle \omega_X c_1, \varepsilon \cdot e_i \cdot s \rangle \cdot
     \langle c_2, s \rangle + \langle c_1, \varepsilon \cdot e_i \cdot s
     \rangle \cdot \langle \omega_X c_2, s \rangle . \]
  \begin{align}
    & \nabla_i Z_j \nonumber\\
    = & - \tfrac{1}{2} \langle \omega_X c_1, \varepsilon \cdot e_j \cdot
    \varepsilon \cdot q (e_i) \cdot \omega_{N_0} s \rangle \cdot \langle c_2,
    s \rangle - \tfrac{1}{2} \langle \omega_X c_1, \varepsilon \cdot e_j \cdot
    s \rangle \cdot \langle c_2, \varepsilon \cdot q (e_i) \cdot \omega_{N_0}
    s \rangle \nonumber\\
    & \quad - \tfrac{1}{2} \langle c_1, \varepsilon \cdot e_j \cdot
    \varepsilon \cdot q (e_i) \cdot \omega_{N_0} s \rangle \cdot \langle
    \omega_X c_2, s \rangle - \tfrac{1}{2} \langle c_1, \varepsilon \cdot e_j
    \cdot s \rangle \cdot \langle \omega_X c_2, \varepsilon \cdot q (e_i)
    \cdot \omega_{N_0} s \rangle \nonumber\\
    = & \tfrac{1}{2} \langle \omega_X c_1, \varepsilon \cdot e_j \cdot
    \varepsilon \cdot q (e_i) \cdot s \rangle \cdot \langle c_2, s \rangle +
    \tfrac{1}{2} \langle \omega_X c_1, \varepsilon \cdot e_j \cdot s \rangle
    \cdot \langle c_2, \varepsilon \cdot q (e_i) \cdot s \rangle \nonumber\\
    & \quad - \tfrac{1}{2} \langle c_1, \varepsilon \cdot e_j \cdot
    \varepsilon \cdot q (e_i) \cdot s \rangle \cdot \langle \omega_X c_2, s
    \rangle - \tfrac{1}{2} \langle c_1, \varepsilon \cdot e_j \cdot s \rangle
    \cdot \langle \omega_X c_2, \varepsilon \cdot q (e_i) \cdot s \rangle
    \nonumber\\
    = & \tfrac{1}{2} \langle \omega_X c_1, (\varepsilon \cdot e_j \cdot
    \varepsilon \cdot q (e_i) + \varepsilon \cdot q (e_i) \cdot \varepsilon
    \cdot e_j) \cdot s \rangle \cdot \langle c_2, s \rangle \nonumber\\
    & \quad - \tfrac{1}{2} \langle c_1, (\varepsilon \cdot e_j \cdot
    \varepsilon \cdot q (e_i) + \varepsilon \cdot q (e_i) \cdot \varepsilon
    \cdot e_j) \cdot s \rangle \cdot \langle \omega_X c_2, s \rangle -
    \tfrac{1}{2} \langle c_1, \varepsilon \cdot e_j \cdot s \rangle \cdot
    \langle \omega_X c_2, \varepsilon \cdot q (e_i) \cdot s \rangle
    \nonumber\\
    = & \langle e_j, q (e_i) \rangle z = q_{i j} z. \nonumber
  \end{align}
  
  Hence,
  \[ \nabla_i |z|^2 = \bar{z} \nabla_i z + z \nabla_i \bar{z} = q_{i j}
     (\bar{z} Z_j + z \bar{Z}_j) . \]
  And
  \[ \nabla_i |Z|^2 = \bar{Z}_j \nabla_i Z_j + Z_j \nabla_i \bar{Z}_j = q_{i
     j} (\bar{Z}_j z + Z_j \bar{z}) . \]
  So $\mathrm{d} (|z|^2 - |Z|^2) = 0$. Hence
  \[ |Z| = |z| \text{ on } \Omega . \]
  Because that $|z| = 0$ along a face $F$ by Lemma \ref{lm:commute omega X},
  so using the above, $|Z| = 0$ along $F$ as well. Also, we easily have the
  estimate
  \[ | \nabla (|Z|  + |z| ) | \leqslant C_0 (|z| + |Z|) \]
  for some constant $C_0 > 0$.
  
  Now we take a unit speed curve $l : [0, l_0] \to \Omega$ parameterized by
  $\tau$, we set the function $\zeta (\tau) : [0, l_0]$ to be the value of the
  function $|Z| + |z|$ at $l (\tau)$, then by the above ordinary differential
  inequality, we have
  \[ | \zeta' (\tau) | \leqslant C_0 | \zeta (\tau) | \text{ for } \tau \in
     [0, l_0] \]
  for the same constant $C_0$. Hence
  \[ 0 \leqslant \zeta (\tau) \leqslant \exp (C_0 \tau) \zeta (0) = 0, \]
  giving $\zeta (\tau) \equiv 0$ for all $\tau \in [0, l_0]$. Since the curve
  $l$ can be arbitrarily taken, so we conclude that
  \[ z = 0, \text{ } Z = 0 \text{ in } \Omega . \]
  As $F$ run through all faces, $X$ run through all vectors of the form $N -
  \langle N_0, N \rangle N_0$. When $X$ is taken to be $N_0$, $z = 0$ follows
  from \eqref{cross orthogonal}. By linearity, the lemma is proved.
\end{proof}

\begin{proposition}
  \label{G id}Let $G = (G_{\alpha \beta})_{1 \leqslant \alpha, \beta \leqslant
  n}$ be the matrix given by $G_{\alpha \beta} = \langle s_{\alpha}, s_{\beta}
  \rangle$. Then $G$ is a non-zero multiple of the identity matrix everywhere.
  Note that $|s_{\alpha} |$ might not be a constant.
\end{proposition}

\begin{proof}
  For any $c_i \in \mathbb{C}^m$, we set $c_i = c_i^+ + c_i^-$ where
  $c_i^{\pm} = \tfrac{1}{2} (1 \pm \omega_{N_0}) c_i$. It is clear that
  $c_i^{\pm} \in \Lambda_{\pm}$. Using Lemma \ref{communte with omega X plus
  minus}, \eqref{cross orthogonal} and the decomposition $c_i = c_i^+ + c_i^-$
  in
  \[ \langle \omega_X c_1, s \rangle \cdot \langle c_2, s \rangle - \langle
     c_1, s \rangle \cdot \langle \omega_X c_2, s \rangle, \]
  we can verify that
  \[ \langle \omega_X c_1, s \rangle \cdot \langle c_2, s \rangle = \langle
     c_1, s \rangle \cdot \langle \omega_X c_2, s \rangle \]
  for all Euclidean vectors $X$. The above says that $G$ commutes with any
  $\omega_X$ which implies that $G$ is a multiple of the identity. Indeed, we
  write the above in components. Let $c_i = (c_1^{(1)}, \ldots, c_1^{(m)})$,
  we have
\begin{align}
\langle \omega_X c_1, s \rangle \cdot \langle c_2, s \rangle & = \langle
c_1, \omega_X s \rangle \cdot \langle c_2, s \rangle = \langle
\bar{c}_1^{(\alpha)} \omega_{X \alpha \lambda} s_{\lambda},
\bar{c}_2^{(\mu)} s_{\mu} \rangle \\
& = \bar{c}_1^{(\alpha)} \omega_{X \alpha \lambda} G_{\lambda \mu}
c_2^{(\mu)}, \\
\langle c_1, s \rangle \cdot \langle \omega_X c_2, s \rangle & = \langle
c_1, s \rangle \cdot \langle c_2, \omega_X s \rangle = \langle
\bar{c}_1^{(\alpha)} s_{\alpha}, \bar{c}_2^{\mu} \omega_{X \mu \lambda}
s_{\lambda} \rangle \\
& = \bar{c}_1^{(\alpha)} G_{\alpha \lambda} \bar{\omega}_{X \mu \lambda}
c_2^{(\mu)} .
\end{align}
  Since $c_1$ and $c_2$ are arbitrary, we know
  \[ \omega_{X \alpha \lambda} G_{\lambda \mu} = G_{\alpha \lambda}
     \bar{\omega}_{X \mu \lambda} . \]
  Since $\omega_X$ is self-adjoint and $\omega_X^2 = 1$, $\bar{\omega}_{X \mu
  \lambda} = \omega_{X \lambda \mu}$ and hence $\omega_{X \alpha \lambda}
  G_{\lambda \mu} = G_{\alpha \lambda} \omega_{\lambda \mu}$. That is, $G$ and
  $\omega_X$ commute.
\end{proof}

\begin{proposition}
  \label{lm:identifcation W}Given any $\omega_{N_0} c_1 = c_1$, $\omega_{N_0}
  c_2 = \pm c_2$ if
  \[ z = \langle c_1, s \rangle \cdot \langle c_1, s \rangle - \langle c_2, s
     \rangle \cdot \langle c_2, s \rangle = 0, \]
  then
  \[ \langle c_1, \varepsilon \cdot e_i \cdot s \rangle \cdot \langle c_1, s
     \rangle \mp \langle c_2, \varepsilon \cdot e_i \cdot s \rangle \cdot
     \langle c_2, s \rangle = 0 \]
  for any unit tangent vectors $e_i \in T \Omega$.
\end{proposition}

\begin{proof}
  We only prove for the case $\omega_{N_0} c_2 = c_2$. \ The proof for the
  case $\omega_{N_0} c_2 = - c_2$ is similar. Let $Z = Z_i e_i$ be the vector
  field given by
  \[ Z_i = \langle c_1, \varepsilon \cdot e_i \cdot s \rangle \cdot \langle
     c_1, s \rangle - \langle c_2, \varepsilon \cdot e_i \cdot s \rangle \cdot
     \langle c_2, s \rangle . \]
  Since at $x_0$, $\varepsilon \cdot \xi \cdot \omega_{N_0} s = s$, $\langle
  Z, \xi \rangle = z = 0$ at $x_0$. Also at $x_0$, as
  \[ \langle c_j, \varepsilon \cdot \xi \cdot s \rangle \cdot \langle c_j, s
     \rangle = \langle c_j, s \rangle \cdot \langle c_j, s \rangle, \]
  by Remark \ref{rk:f geq W}, for any vector $e$ orthogonal to $\xi$,
  \[ \langle c_j, \varepsilon \cdot e \cdot s \rangle \cdot \langle c_j, s
     \rangle = 0 \]
  for $j = 1, 2$. So $Z$ vanishes at $x_0$. Note that $z$ and $Z^i$ are real,
  it follows from \eqref{grad of W} and linearity that
  \[ \nabla_i |Z|^2 = 2 \langle \nabla_i Z, Z \rangle = 2 Z_k \nabla_i Z^k = -
     2 z Z_k q_{i k} = 0. \]
  We conclude that $|Z|$ is a constant, hence $Z$ vanishes since $Z$ vanishes
  at $x_0$.
\end{proof}

\begin{remark}
  \label{rk:flat}The lemma gives that $\langle c, s \rangle$, $c \in
  \mathbb{C}^m$ determines the same $W / f$, and hence $\Sigma_t$. Indeed, let
  $j = 1, 2$, $f^{(j)} = | \psi_j |^2$ and $W^{(j)} = \langle \varepsilon
  \cdot e_i \cdot \psi_j, \psi_j \rangle$, where $\psi_j = \langle c_j, s
  \rangle$ and $\omega_{N_0} c_j = c_j$. So $W^{(1)} / f^{(1)} = W^{(2)} /
  f^{(2)}$. Let $\langle c_i, s \rangle = \psi_i$, then by Lemma
  \ref{lm:foliation with parallel spinor}, $\{\psi_i | \psi_i |^{- 1} \}$ \ is
  a basis of parallel spinor along $\Sigma_t$. So $\Sigma_t$ is flat.
\end{remark}

\subsection{Analysis of Gauss and Codazzi equation}\label{sec:structure}

Let $\{e_i \}_{i = 1, \ldots, n}$ be a geodesic normal frame. We compute the
curvatures of $(M, g)$. Since $f = |W|$, we can set $e_n = W / f$. The Lemma
\ref{lm:one spinor omega N0 identification} and Proposition
\ref{lm:identifcation W} identifies the matrix $\omega_{N_0}$ action on $s$
with the Clifford multiplication by $\varepsilon \cdot \tfrac{W}{f}$, so
\begin{equation}
  \omega_{N_0} s = \varepsilon \cdot e_n \cdot s. \label{eq:identification w
  N0}
\end{equation}
We differentiate \eqref{eq:killing} in the $e_j$ direction, use
\eqref{eq:identification w N0} and obtain that
\begin{align}
0 = & \nabla_j (\nabla_i s + \tfrac{1}{2} \varepsilon \cdot q (e_i) \cdot
\omega_{N_0} s) \\
= & \nabla_j \nabla_i s + \tfrac{1}{2} \varepsilon \cdot (\nabla_j q) (e_i)
\cdot \omega_{N_0} s + \tfrac{1}{2} \varepsilon \cdot q (e_i) \cdot
\omega_{N_0} (- \tfrac{1}{2} \varepsilon \cdot q (e_j) \omega_{N_0} s)
\\
= & \nabla_j \nabla_i s + \tfrac{1}{2} \varepsilon \cdot (\nabla_j q) (e_i)
\cdot \varepsilon \cdot e_n \cdot s + \tfrac{1}{4} q (e_i) \cdot q (e_j)
\cdot s \\
= & \nabla_j \nabla_i s - \tfrac{1}{2} (\nabla_j q) (e_i) \cdot e_n \cdot s
+ \tfrac{1}{4} q (e_i) \cdot q (e_j) \cdot s.
\end{align}
Switching roles of $i$ and $j$, and subtracting, we see
\begin{align}
\tfrac{1}{4} R_{i j k l} e_k \cdot e_l \cdot s = & \nabla_i \nabla_j s -
\nabla_j \nabla_i s \\
= & \tfrac{1}{2} (\nabla_i q) (e_j) \cdot e_n \cdot s - \tfrac{1}{2}
(\nabla_j q) (e_i) \cdot e_n \cdot s \\
& \quad - \tfrac{1}{4} q (e_j) \cdot q (e_i) \cdot s + \tfrac{1}{4} q (e_i)
\cdot q (e_j) \cdot s \\
= & \tfrac{1}{2} (\nabla_i q_{j k} - \nabla_j q_{i k}) e_k \cdot e_n \cdot s
\\
& \quad - \tfrac{1}{4} q_{j k} q_{i l} e_k \cdot e_l \cdot s + \tfrac{1}{4}
q_{i k} q_{j l} e_k \cdot e_l \cdot s.
\end{align}
\

That is,
\begin{equation}
  \hat{R}_{i j k l} e_k \cdot e_l \cdot s = 2 (\nabla_i q_{j k} - \nabla_j
  q_{i k}) e_k \cdot e_n \cdot s, \label{curvature relation with s}
\end{equation}
where for brevity, we have set
\begin{equation}
  \hat{R}_{i j k l} : = R_{i j k l} + q_{j k} q_{i l} - q_{i k} q_{j l} .
  \label{R hat}
\end{equation}
Let
\begin{equation}
  \tau_{k l} = (\nabla_i q_{j k} - \nabla_j q_{i k}) \delta_{l n} \label{tau k
  l}
\end{equation}
and note that $\tau_{k l} \neq \tau_{l k}$ in general. We have
\begin{align}
\tau_{k l} e_k \cdot e_l = & \tau_{k l} (- 2 \delta_{k l} - e_l \cdot e_k)
\\
& = - 2 \sum_l \tau_{l l} - \tau_{l k} e_k \cdot e_l \\
& = - 2 \tau_{n n} - \tau_{l k} e_k \cdot e_l,
\end{align}
and so
\begin{equation}
  \hat{R}_{i j k l} e_k \cdot e_l \cdot s = \tau_{k l} e_k \cdot e_l \cdot s +
  \tau_{k l} e_k \cdot e_l \cdot s = - 2 \tau_{n n} s + (\tau_{k l} - \tau_{l
  k}) e_k \cdot e_l \cdot s. \label{R hat tau relation}
\end{equation}
Because $n$ is even, the representation $c : \ensuremath{\operatorname{Cl}}
(T_x \Omega, \mathbb{C}) \to \ensuremath{\operatorname{End}} (S_{\Omega_g}
|_x)$ is an isomorphism (see {\cite[Theorem
1.28]{bourguignon-spinorial-2015}}), we conclude that
\begin{equation}
  \tau_{n n} = \nabla_i q_{j n} - \nabla_j q_{i n} = 0 \label{t n n}
\end{equation}
and that
\begin{equation}
  \hat{R}_{i j k l} - \tau_{k l} + \tau_{l k} = 0. \label{eq:gauss like}
\end{equation}
We have the following consequences of the identity \eqref{eq:gauss like}
depending on the values of $k$ and $l$. Essentially, there are two cases.

\text{{\bfseries{Case I}}}, $k < n$ and $l < n$. We have
\begin{equation}
  \hat{R}_{i j k l} = R_{i j k l} + q_{j k} q_{i l} - q_{i k} q_{j l} = 0.
  \label{eq:all tangential}
\end{equation}

\text{{\bfseries{Case II}}}, $k < n$ and $l = n$. We have
\begin{equation}
  \hat{R}_{i j k n} - \tau_{k n} = R_{i j k n} + q_{j k} q_{i n} - q_{i k}
  q_{j n} - (\nabla_i q_{j k} - \nabla_j q_{i k}) = 0. \label{eq:one normal}
\end{equation}

Let $i = n$ and renaming of indices in \eqref{eq:all tangential}, we see that
$\hat{R}_{i j k n} = 0$ for $i < n$, $j < n$, $k < n$. So
\[ \tau_{k n} = \nabla_i q_{j k} - \nabla_j q_{i k} = 0 \]
for $i < n$, $j < n$, $k < n$. by \eqref{eq:one normal}. We see then by
\eqref{eq:all tangential} and \eqref{eq:one normal} that
\begin{align}
2 \mu = & \sum_{i, j} \hat{R}_{i j j i} \\
= & \sum_{j \neq n} (\hat{R}_{n j j n} + \widehat{R }_{j n n j}) \\
= & \sum_{j \neq n} 2 \hat{R}_{n j j n} \\
= & 2 \sum_{j \neq n} (\nabla_n q_{j j} - \nabla_j q_{n j}) = - 2 J (e_n) .
\end{align}
\

\subsection{Boundary curvatures}

\begin{lemma}
  \label{lm:boundary second fundamental form}Let $F_{\ell}$ be a face of
  $\Omega$, then the second fundamental form $h$ of $F_{\ell}$ is given by
  \begin{equation}
    h_{i k} + \cos \theta_{\ell} q_{i k} = q (e_i, \nu_{\ell}) \langle \xi,
    e_k \rangle . \label{eq:boundary 2ff}
  \end{equation}
\end{lemma}

\begin{proof}
  For brevity, we suppress the subscript on $\nu_{\ell}$ and $N_{\ell}$. Now
  we set $\xi = W / f$, so $\omega_{N_0} s = \varepsilon \cdot \xi \cdot s$.
  (Previously, $\xi$ is only defined at $x_0$.) We differentiate the boundary
  condition \eqref{eq:bdry condition},
  \[ \nabla_i (\varepsilon \cdot \nu \cdot \omega_N s) = \nabla_i s = -
     \tfrac{1}{2} \varepsilon \cdot q (e_i) \cdot \omega_{N_0} s. \]
  The left is
  \[ \varepsilon \cdot h_{i k} e_k \cdot \omega_N s + \varepsilon \cdot \nu
     \cdot \omega_N (- \tfrac{1}{2} \varepsilon \cdot q (e_i) \cdot
     \omega_{N_0} s), \]
  so
  \[ \varepsilon \cdot h_{i k} e_k \cdot \omega_N s + \varepsilon \cdot \nu
     \cdot \omega_N (- \tfrac{1}{2} \varepsilon \cdot q (e_i) \cdot
     \omega_{N_0} s) = - \tfrac{1}{2} \varepsilon \cdot q (e_i) \cdot
     \omega_{N_0} s. \]
  Applying $\varepsilon \omega_N$ on both sides, we obtain
  \[ h_{i k} e_k \cdot s - \tfrac{1}{2} \nu \cdot \varepsilon \cdot q (e_i)
     \omega_{N_0} s = - \tfrac{1}{2} q (e_i) \cdot \omega_N \omega_{N_0} s. \]
  Using $\omega_{N_0} s = \varepsilon \cdot \xi \cdot s$ and then $\omega_N s
  = \varepsilon \cdot \nu \cdot s$, we see
  \[ h_{i k} e_k \cdot s - \tfrac{1}{2} \nu \cdot \varepsilon \cdot q (e_i)
     \cdot \varepsilon \cdot \xi \cdot s = - \tfrac{1}{2} q (e_i) \cdot
     \varepsilon \cdot \xi \cdot \varepsilon \cdot \nu \cdot s. \]
  So
  \[ h_{i k} e_k \cdot s + \tfrac{1}{2} \nu \cdot q (e_i) \cdot \xi \cdot s =
     \tfrac{1}{2} q (e_i) \cdot \xi \cdot \nu \cdot s. \]
  Let $\xi^{\bot}$ ($\zeta^{\bot}$) be the component of $\xi$ ($\zeta : = q
  (e_i)$) normal to a face $F_{\ell}$ and $\xi^{\top} = \xi - \xi^{\bot}$
  ($\zeta^{\top} = \zeta - \zeta^{\bot}$). We get
  \[ h_{i k} e_k \cdot s + \tfrac{1}{2} \nu \cdot (\zeta^{\bot} +
     \zeta^{\top}) \cdot (\xi^{\bot} + \xi^{\top}) \cdot s = \tfrac{1}{2}
     (\zeta^{\bot} + \zeta^{\top}) \cdot (\xi^{\bot} + \xi^{\top}) \cdot \nu
     \cdot s, \]
  and after canceling some terms,
\begin{align}
& h_{i k} e_k \cdot s + \tfrac{1}{2} \nu \cdot \zeta^{\bot} \cdot
\xi^{\top} \cdot s + \tfrac{1}{2} \nu \cdot \zeta^{\top} \cdot \xi^{\bot}
\cdot s \\
= & \tfrac{1}{2} \zeta^{\bot} \cdot \xi^{\top} \cdot \nu \cdot s +
\tfrac{1}{2} \zeta^{\top} \cdot \xi^{\bot} \cdot \nu \cdot s.
\end{align}
  Equivalently,
\begin{align}
& h_{i k} e_k \cdot s + \tfrac{1}{2} \nu \cdot \langle \zeta, \nu \rangle
\nu \cdot \xi^{\top} \cdot s + \tfrac{1}{2} \nu \cdot \zeta^{\top} \cdot
\langle \xi, \nu \rangle \nu \cdot s \\
= & \tfrac{1}{2} \langle \zeta, \nu \rangle \nu \cdot \xi^{\top} \cdot \nu
\cdot s + \tfrac{1}{2} \zeta^{\top} \cdot \langle \xi, \nu \rangle \nu
\cdot \nu \cdot s.
\end{align}
  This gives
  \[ h_{i k} e_k \cdot s = - \langle \xi, \nu \rangle \zeta^{\top} \cdot s +
     \langle \zeta, \nu \rangle \xi^{\top} \cdot s. \]
  So
\begin{align}
& 2 h_{i j} |s|^2 \\
= & \langle h_{i k} e_k \cdot s, e_j \cdot s \rangle + \langle e_j \cdot
s, h_{i k} e_k \cdot s \rangle \\
= & \langle - \langle \xi, \nu \rangle \zeta^{\top} \cdot s + \langle
\zeta, \nu \rangle \xi^{\top} \cdot s, e_j \cdot s \rangle \\
& \quad + \langle e_j \cdot s, - \langle \xi, \nu \rangle \zeta^{\top}
\cdot s + \langle \zeta, \nu \rangle \xi^{\top} \cdot s \rangle
\\
= & {(- 2 \langle \xi, \nu \rangle \langle \zeta, e_j}  \rangle + 2
\langle \zeta, \nu \rangle \langle \xi, e_j \rangle) |s|^2 .
\end{align}
  Finally, as $|s| \neq 0$,
  \[ h_{i j} = - \langle \xi, \nu \rangle q_{i j} + q (e_i, \nu) \langle \xi,
     e_j \rangle . \]
  
\end{proof}

\begin{remark}
  Using symmetry of $h$ and $q$ in \eqref{eq:boundary 2ff}, we see
  \[ q (e_i, \nu) \langle \xi, e_k \rangle = q (e_k, \nu) \langle \xi, e_i
     \rangle . \]
  Assume that $e_k$ is normal to $\Sigma_t \cap F_{\ell}$ and tangent to
  $F_{\ell}$ , $e_i$ is tangent to $\Sigma_t \cap F_{\ell}$, then $q (e_i,
  \nu) = 0$.
\end{remark}

\begin{lemma}
  \label{lm:boundary geodesic}If for the face $F_{\ell}$, $\theta_{\ell}$ is
  neither 0 nor $\pi$, then $\partial_{(\ell)} \Sigma_t = F_{\ell} \cap
  \Sigma_t$ is totally geodesic in $\Sigma_t$.
\end{lemma}

\begin{proof}
  Let $\{e_i \}$ be an orthonormal frame such that $e_n = \xi$, $e_{n - 1}$
  normal to $\partial_{(\ell)} \Sigma_t$ and ${\{e_i \}_{i = 1, \ldots, n -
  2}} $ tangent to $\partial_{(\ell)} \Sigma_t$. Let $A_{\partial_{(\ell)}
  \Sigma_t}$ be the second fundamental form of $\partial_{(\ell)} \Sigma_t$ in
  $\Sigma_t$. Since $\theta_{\ell}$ is neither 0 nor $\pi$, $\eta := \nu -
  \langle \xi, \nu \rangle \xi$ is a non-zero vector normal to both
  $\partial_{(\ell)} \Sigma_t$ and $\xi$. So $e_{n - 1} = | \eta |^{- 1}
  \eta$. Letting $e_i$ and $e_j$ be tangent to $\partial_{(\ell)} \Sigma_t$,
  we see
\begin{align}
& A_{\partial_{(\ell)} \Sigma_t} (e_i, e_j) \\
= & \langle \nabla_i e_{n - 1}, e_j \rangle \\
= & | \eta |^{- 1} \langle \nabla_i (\nu - \langle \xi, \nu \rangle \xi),
e_j \rangle \\
= & | \eta |^{- 1} (h_{i j} - \langle \xi, \nu \rangle \langle \nabla_i
\xi, e_j \rangle) \\
= & | \eta |^{- 1} (h_{i j} + \langle \xi, \nu \rangle q_{i j})
\\
= & | \eta |^{- 1} q (e_i, \nu) \langle \xi, e_j \rangle = 0,
\end{align}
  where in the last line we have used Lemma \ref{lm:boundary second
  fundamental form}. Hence $\partial_{(\ell)} \Sigma_t = F_{\ell} \cap
  \Sigma_t$ is totally geodesic in $\Sigma_t$.
\end{proof}

\subsection{Proof of Theorem \ref{dihedral ids}}

We are now ready to summarize the proof of Theorem \ref{dihedral ids}.

\begin{proof}[Proof of Theorem \ref{dihedral ids}]
  Assume that the dimension $n$ is even. Item \text{{\itshape{a)}}} follows
  from Lemma \ref{lm:W is gradient}, $\Sigma_t$ is a flat polyhedron follows
  from Remark \ref{rk:flat} and Lemma \ref{lm:boundary geodesic}. Item
  \text{{\itshape{b)}}} follows from Lemma \ref{lm:capillary}. Item
  \text{{\itshape{c)}}} follows from the proof of Lemma \ref{lm:foliation with
  parallel spinor}. Item \text{{\itshape{d)}}} follows from Subsection
  \ref{sec:structure}. For the odd dimensional case, see Subsection \ref{sub
  odd}.
\end{proof}

\section{Applications}\label{sec:extra}

In this section, we show additional relations of curvatures with $q$ from the
assumptions and conclusions of Theorem \ref{thm:egm}, and we comment how to
handle the odd dimensional case of Theorem \ref{dihedral ids}.

\subsection{Additional results following Theorem \ref{thm:egm}}

We give some further consequences from the assumptions from Theorem
\ref{thm:egm}.

\begin{theorem}
  \label{strengthened egm}Let $(M^n, g, q)$ be as in Theorem \ref{thm:egm} .
  Moreover, let $\{e_i \}$ be an orthonormal frame such that $e_n = \nu_t$ and
  \[ \hat{R}_{i j k l} = R_{i j k l} + q_{j k} q_{i l} - q_{i k} q_{j l} . \]
  Then
\begin{align}
\nabla_i q_{j n} - \nabla_j q_{i n} & = 0, \\
\hat{R}_{n j k n} & = \nabla_n q_{j k} - \nabla_j q_{n k}, \\
\hat{R}_{i j k l} & = 0, \label{hz gauss} \\
\hat{R}_{i j k n} & = - (\nabla_i q_{j k} - \nabla_j q_{i k}), \label{R
hat i j k n}
\end{align}
  for all $i < n$, $j < n$, $k < n$, $l < n$.
\end{theorem}

\begin{proof}
  From Theorem \ref{thm:egm}, we know that $M$ is foliated by $\Sigma_t$ with
  \begin{equation}
    (Q^+)_{i j} = h_{i j} + q_{i j} = 0 \label{Q vanish}
  \end{equation}
  along $\Sigma_t$ ($i < n$, $j < n$), where $h_{i j}$ is the second
  fundamental form with respect to the normal pointing to the $\partial_+ M$
  side. Every $\Sigma_t$ is a flat $(n - 1)$-torus. From here on in the proof,
  we assume that the indices $i, j, k, l$ are all less than $n$.
  
  Let the induced metric of $g$ on $\Sigma_t$ be $g_t$, and $\bar{\nabla}$ be
  the Levi-Civita connection of $(\Sigma_t, g_t)$ and we set $e_n = \nu_t$.
  Then
  \begin{equation}
    \mu + J (e_n) = 0, \label{mu J en}
  \end{equation}
  along $\Sigma_t$.
  
  For simplicity, we can assume that at a point $p \in \partial \Sigma_t$
  that the orthonormal frame $\{e_i \}$ is chosen so that $\bar{\nabla}_i e_j
  = 0$ and $\nabla_i e_j = - h_{i j}$. The computations are all done at the
  point $p$. To verify the relation $\hat{R}_{i j k l} = 0$, we use Gauss
  equation, \eqref{Q vanish} and the flatness of $\Sigma_t$,
  \[ \hat{R}_{i j k l} = R_{i j k l} + h_{j k} h_{i l} - h_{i k} h_{j l} = 0.
  \]
  For the proof of the relation $\hat{R}_{i j k n} = - (\nabla_i q_{j k} -
  \nabla_j q_{i k})$, we only have to invoke the Codazzi equation
  $\bar{\nabla}_i h_{j k} - \bar{\nabla}_j h_{i k} = R_{i j k n}$ and the
  following
\begin{align}
& \bar{\nabla}_i q_{j k} \\
= & \nabla_i q_{j k} + q (\nabla_i e_j, e_k) + q (e_j, \nabla_i e_k)
\\
= & \nabla_i q_{j k} - h_{i j} q_{k n} - h_{i k} q_{j n} \\
= & \nabla_i q_{j k} + q_{i j} q_{k n} + q_{i k} q_{j n} .
\end{align}
  It remains to verify $\hat{R}_{n j k n} = \nabla_n q_{j k} - \nabla_j q_{n
  k}$ to finish the proof. Since the foliation $\Sigma_t$ is given by the
  level sets of the function $t$, the metric $g$ can be written as
  \[ g = | \nabla t|^{- 2} \mathrm{d} t^2 + g_t . \]
  We see that $\partial_t = | \nabla t|^{- 1} e_n$. Recall the variational
  formula {\cite{andersson-local-2005}} of the null expansion $\partial_t
  (\ensuremath{\operatorname{tr}}_{\Sigma_t} Q^+)$,
\begin{align}
& \partial_t (\ensuremath{\operatorname{tr}}_{\Sigma_t} Q^+) \\
= & - \bar{\Delta} | \nabla t|^{- 1} + 2 \langle Y, \bar{\nabla} | \nabla
t|^{- 1} \rangle \\
& \qquad + (\tfrac{1}{2} R_{\Sigma_t} - \mu - J (e_n) - \tfrac{1}{2} |Q^+
|^2 +\ensuremath{\operatorname{div}}_{\Sigma_t} Y - |Y|^2) | \nabla t|^{-
1},
\end{align}
  where $Y = q (\nu, \cdot)^{\top}$ and we have used $\bar{\Delta}$ to denote
  the Laplace-Beltrami operator, $R_{\Sigma_t}$ the scalar curvature and
  $\ensuremath{\operatorname{div}}_{\Sigma_t}$ the divergence of a vector
  field on $\Sigma_t$. For $\Sigma_t$, we see
  \[ 0 = - \bar{\Delta} | \nabla t|^{- 1} + 2 \langle Y, \bar{\nabla} | \nabla
     t|^{- 1} \rangle + (\ensuremath{\operatorname{div}}_{\Sigma_t} Y - |Y|^2)
     | \nabla t|^{- 1} . \]
  Multiplying the above by $- | \nabla t|$, and integration by parts on
  $\Sigma_t$ leads to
  \[ \int_{\Sigma_t} | \bar{\nabla} \log | \nabla t| + Y|^2 = \int_{\Sigma_t}
     (| \nabla t| \bar{\Delta} | \nabla t|^{- 1} - 2| \nabla t| \langle Y,
     \bar{\nabla} \log | \nabla t|^{- 1} \rangle + |Y|^2) = 0. \]
  It follows immediately $\bar{\nabla} \log | \nabla t| = - Y$ (see also
  {\cite[(4.13)]{tsang-dihedral-2021-arxiv}}). For $e_i$, we see
  \begin{equation}
    \bar{\nabla}_i \log | \nabla t| = - q (e_i, e_n) . \label{eq:tangential
    qn}
  \end{equation}
  We get $\bar{\nabla}_i q_{j n} = \bar{\nabla}_j q_{i n}$ along $\Sigma_t$.
  By \eqref{Q vanish},
\begin{align}
& \bar{\nabla}_i q_{j n} \\
= & \nabla_i q_{j n} + q (\nabla_i e_j, e_n) + q (e_j, \nabla_i e_n)
\\
= & \nabla_i q_{j n} - h_{i j} q_{n n} + \sum_{k \neq n} q_{j k} h_{i k}
\\
= & \nabla_i q_{j n} + q_{i j} q_{n n} - \sum_{k \neq n} q_{j k} q_{i k},
\label{eq:tangential grad of qn}
\end{align}
  and symmetries of $i$ and $j$, we see
  \[ \nabla_i q_{j n} - \nabla_j q_{i n} = 0. \]
  Now we can compute $\partial_t (h_{j k} + q_{j k}) = 0$ to get the relation
  \begin{equation}
    \hat{R}_{n j k n} = \nabla_n q_{j k} - \nabla_j q_{n k} . \label{untraced
    mu J}
  \end{equation}
  Indeed, using \eqref{Q vanish} and \eqref{eq:tangential qn}, we see
\begin{align}
& \partial_t h_{j k} \\
= & - \bar{\nabla}_j \bar{\nabla}_k | \nabla t|^{- 1} - R_{n j k n} |
\nabla t|^{- 1} + h_{i k} h_{j k} | \nabla t|^{- 1} \\
= & - \bar{\nabla}_j \log | \nabla t| \bar{\nabla}_k \log | \nabla t| +
\bar{\nabla}_j \bar{\nabla}_k \log | \nabla t| - R_{n j k n} | \nabla
t|^{- 1} + h_{i k} h_{j k} | \nabla t|^{- 1} \\
= & | \nabla t|^{- 1} \left( - q_{j n} q_{k n} - \bar{\nabla}_j q_{k n} -
R_{n j k n} + \sum_{i \neq n} q_{i j} q_{i k} \right) . \label{t variation
of 2ff}
\end{align}
  And using \eqref{Q vanish} and \eqref{eq:tangential qn} again,
\begin{align}
& \partial_t q_{j k} \\
= & | \nabla t|^{- 1} \nabla_n q_{j k} + q (\nabla_{\partial_t} e_j, e_k)
+ q (e_j, \nabla_{\partial_t} e_k) \\
= & | \nabla t|^{- 1} \nabla_n q_{j k} + q (\nabla_j \partial_t, e_k) + q
(e_j, \nabla_k \partial_t) \\
= & | \nabla t|^{- 1} \nabla_n q_{j k} + \nabla_j | \nabla t|^{- 1} q_{k
n} + | \nabla t|^{- 1} q (\nabla_j e_n, e_k) \\
& \qquad + \nabla_k | \nabla t|^{- 1} q_{j n} + | \nabla t|^{- 1} q
(\nabla_k e_n, e_j) \\
= & | \nabla t|^{- 1} \left( \nabla_n q_{j k} + 2 q_{j n} q_{k n} - 2
\sum_{i \neq n} q_{i j} q_{i k} \right) .
\end{align}
  The relation \eqref{untraced mu J} now follows by simply combining
  \eqref{eq:tangential grad of qn}, \eqref{t variation of 2ff} and the above.
\end{proof}

\begin{remark}
  \label{not full relation}It follows from $\mu \geqslant |J|$ and $\mu + J
  (e_n) = 0$ that $J (e_i) = 0$. We obtain
  \[ \sum_{i j j \nu} \hat{R}_{i j j \nu} = J (e_i) = 0 \]
  by \eqref{R hat i j k n}. However, we are not able to show that
  \[ \nabla_i q_{j k} - \nabla_j q_{i k} = 0 \]
  for all $i < n$, $j < n$, $k < n$. This is a main difference from Theorem
  \ref{dihedral ids}. Also, it is an interesting question to seek a spinorial
  proof of Theorem \ref{strengthened egm}.
\end{remark}

\subsection{Comments on odd-dimensional case}\label{sub odd}

For the odd-dimensional case, we use the following connection and Dirac
operator
\begin{equation}
  \tilde{\nabla}_{e_i} \sigma = \nabla_{e_i} \sigma + \tfrac{1}{2} (c (\sqrt{-
  1} q (e_i)) \otimes \bar{c} (\sqrt{- 1} N_0)) \sigma,
\end{equation}
where $\sigma$ is a section of $S_{M_g} \otimes S_{M_{\delta}}$ and we have
used the Einstein summation convention and the Dirac operator is
\begin{equation}
  \tilde{D} = D + \tfrac{1}{2} \ensuremath{\operatorname{tr}}_g q \bar{c}
  (N_0) .
\end{equation}
And we use the boundary condition
\[ \chi_{\lambda} \sigma = c (\sqrt{- 1} \nu_{\lambda}) \otimes \bar{c}
   (\sqrt{- 1} N_{\lambda}) \sigma = \sigma \]
on the boundary $\partial \Omega_{\lambda}$ of the approximation
$\Omega_{\lambda}$ of the polyhedron $\Omega$. For the index theory to
guarantee the existence of such a section $\sigma$, we use {\cite[Proposition
2.15]{brendle-scalar-2024}}. The definition of the matrix $\omega_X$ in
\eqref{omega X} is changed to
\[ (\omega_X s)_{\alpha} = \sum_{\beta = 1}^m \omega_{X \alpha \beta}
   s_{\beta} : = \sum_{\beta = 1}^m \langle \sqrt{- 1} \bar{c} (X)
   \bar{s}_{\beta}, \bar{s}_{\alpha} \rangle s_{\beta} . \]
With the rigidity analysis of Section \ref{sec:ids rigidity for polyhedra}
applying these changes and we obtain \eqref{R hat tau relation}. The arguments
from \eqref{R hat tau relation} to arrive \eqref{t n n} and \eqref{eq:gauss
like} are slightly different since the dimension is odd. We know from \eqref{R
hat tau relation} and that the components of $s$ are linearly independent that
\[ \hat{R}_{i j k l} e_k \cdot e_l \cdot + 2 \tau_{n n} - (\tau_{k l} -
   \tau_{l k}) e_k \cdot e_l \cdot \]
lies in the kernel of the spinor representation. Since the dimension $n$ is
odd, the kernel of the spinor representation $c :
\ensuremath{\operatorname{Cl}} (T_x \Omega) \to
\ensuremath{\operatorname{End}} (S_{\Omega_g} |_x)$ is given by the $(-
1)$-eigenspaces of the complex volume element $\varepsilon = (\sqrt{- 1})^{(n
+ 1) / 2} c (e_1) \cdots c (e_n) \in \ensuremath{\operatorname{Cl}} (T_x
\Omega)$ (see Theorem 1.28 and Definition 1.31 of
{\cite{bourguignon-spinorial-2015}}). Hence,
\[ (\hat{R}_{i j k l} e_k \cdot e_l \cdot + 2 \tau_{n n} - (\tau_{k l} -
   \tau_{l k}) e_k \cdot e_l \cdot) (1 + \varepsilon) \]
vanishes and the identities \eqref{t n n} and \eqref{eq:gauss like} then
follows. And we finish the proof for the odd-dimensional case.

\bibliographystyle{alpha}
\bibliography{initial-data-set-polytope}

\end{document}